\documentclass[12pt, twoside]{article}

\usepackage{amsmath,amsthm,amssymb}

\usepackage{times}
\usepackage{xurl}
\usepackage{enumerate}
\usepackage{tikz-cd}
\usepackage{float}
\usepackage{xcolor}
\usepackage{arydshln}

\usepackage{mathtools}
\mathtoolsset{centercolon}

\makeatletter
\def\author#1{\gdef\autrun{\def\and{\unskip, }#1}\gdef\@author{#1}}

\makeatother

\newcommand{\rref}[2]{\hyperref[#2]{{#1}~\ref*{#2}}}

\newtheorem{theorem}{\sc Theorem}%[section]
\newtheorem{lemma}[theorem]{\sc Lemma}
\newtheorem{corollary}[theorem]{\sc Corollary}

\theoremstyle{definition}
\newtheorem{definition}[theorem]{\sc Definition}
\newtheorem{example}[theorem]{\sc Example}
\newtheorem{remark}[theorem]{\sc Remark}

\newcommand{\PG}{\mathrm{PG}}
\renewcommand{\P}{\mathcal{P}}
\renewcommand{\S}{\mathcal{S}}
\newcommand{\Q}{\mathcal{Q}}
\newcommand{\N}{\mathcal{N}}
\newcommand{\C}{\mathcal{C}}
\newcommand{\X}{\mathcal{X}}
\newcommand{\FF}{\mathbb{F}}
\newcommand{\CC}{\mathbb{C}}
\newcommand{\ZZ}{\mathbb{Z}}
\newcommand{\id}{I}

\renewcommand{\phi}{\varphi}

\usepackage{physics,tikz}

\usepackage{hyperref} %last

\begin{document}

\baselineskip=17pt
\author{Simeon Ball\thanks{Supported by PID2020-113082GB-I00 financed by MCIN / AEI / 10.13039/501100011033, the Spanish Ministry of Science and Innovation.}, Edgar Moreno\thanks{Supported by the Higher Interdisciplinary  Training Center (CFIS) - Universitat Politècnica de Catalunya (UPC) and SANTANDER- UPC INITIAL RESEARCH GRANTS (INIREC).} and Robin Simoens\thanks{Supported by the Research Foundation Flanders through the grant 11PG724N.}%\thanks{\texttt{Robin.Simoens@UGent.be}, Department of Mathematics: Analysis, Logic and Discrete Mathematics, Ghent University, Belgium\\ Department of Mathematics, Universitat Politècnica de Catalunya, Spain}
}
\title{Stabiliser codes over fields of even order}
\date{}
\maketitle

%\subjclass{51E21, 15A03, 94B05, 05B35.}

\begin{abstract}
    We prove that the natural isomorphism \(\FF_{2^h}\cong\FF_2^h\) induces a bijection between stabiliser codes on \(n\) quqits with local dimension \(q=2^h\) and binary stabiliser codes on \(hn\) qubits.
    This allows us to describe these codes geometrically: a stabiliser code over a field of even order corresponds to a so-called \emph{quantum set of symplectic polar spaces}. %, a set of \(hn\) lines in a binary projective space that can be partitioned into \(n\) sets of \(h\) lines that are in general position, with the property that any codimension two space is skew to an even number of its lines. 
    Moreover, equivalent stabiliser codes have a similar geometry, which can be used to prove the uniqueness of a \([\![4,0,3]\!]_4\) stabiliser code and the nonexistence of both a \([\![7,1,4]\!]_4\) and an \([\![8,0,5]\!]_4\) stabiliser code.
\end{abstract}

\section{Introduction}

%Using the  geometrical setting of qubit stabiliser codes, in which this code is represented by a set of lines in a binary finite projective space, as introduced in \cite{GGMG}, this allows us to determine the minimum distance of these stabiliser codes geometrically. We also establish that equivalent stabiliser codes have a similar geometry. %In this geometrical setting, the hyperplane intersections determine the weight distribution of the elements in the stabiliser subgroup. This is a useful tool in distinguishing nonequivalent stabiliser groups.

%We extend a result of . They proved that a binary stabiliser code corresponds to a so-called \emph{quantum set of lines}, a set of lines in a projective space with a specific property. We prove that a stabiliser code of characteristic \(2^h\) corresponds to a so-called \emph{quantum set of \(h\)-sets of lines}, a quantum set of lines that can be partitioned into sets of \(h\) lines that are in general position.

It has been known since the late nineties \cite{CRSS} that qubit stabiliser codes are equivalent to classical additive codes contained in their symplectic dual. It was proved in \cite{ketkar} that stabiliser codes over finite fields are equivalent to additive codes over finite fields which are contained in their trace symplectic dual. In this article, we look at stabiliser codes over finite fields of characteristic two in detail. We prove that there is a bijection between stabiliser codes on $n$ quqits over the field ${\mathbb F}_{2^h}$ and stabiliser codes on $hn$ qubits over the field ${\mathbb F}_2$. We prove that the geometrical object which best describes a stabiliser code over a finite field of even order is a set of $n$ symplectic polar spaces of rank $h$ in a projective space over \(\FF_2\). More precisely, a stabiliser code on $n$ quqits with \(q=2^h\), with a stabiliser group of size $2^r$, is equivalent to a set \(\X\) of $n$ projective $(2h-1)$-spaces in a projective $(r-1)$-space each equipped with a symplectic polarity, with the following property: every codimension two subspace intersects an even number of the elements of \(\X\) in a subspace for which the perpendicular space (with respect to the symplectic polarity) is totally isotropic. This allows us to deduce new results concerning  stabiliser codes over ${\mathbb F}_4$. We prove that the $[\![4, 0, 3]\!]_4$ code is unique, that there is no $[\![7, 1, 4]\!]_4$ code and no $[\![8, 0, 5]\!]_4$ code. This last case is of particular interest since it rules out the existence of a \emph{stabilised} absolutely maximum entangled state on $8$ quqits with local dimension $4$, see \cite{Hubertables}.

This paper is structured as follows. In \rref{Section}{sec:prelim}, we give a background on stabiliser codes, and make the link with classical additive codes. In \rref{Section}{sec:main}, we establish the link between stabiliser codes over \(\FF_{2^h}\) and over \(\FF_2\). In \rref{Section}{sec:geometry}, we provide a geometrical interpretation of stabiliser codes over fields of even order. Finally, in \rref{Section}{sec:applications}, these results are used to prove the uniqueness of a \([\![4,0,3]\!]_4\) code and the nonexistence of a \([\![7,1,4]\!]_4\) code.

\section{Preliminaries}\label{sec:prelim}

In order to define stabiliser codes over finite fields of even order, we first need the definition of the Pauli group.

\subsection{The Pauli group}

Let \(n\in\mathbb{N}\) and let \(q=2^h\) be a power of two. Consider \(\mathcal{H}=\left(\CC^q\right)^{\otimes n}\). This Hilbert space provides a mathematical interpretation for \(n\) physical particles that are each in a superposition of \(q\) basis states (quqits). Its elements are linear combinations of \(n\)-dimensional tensors of \(q\)-dimensional complex vectors. A \emph{quantum error-correcting code} is a subspace of \(\mathcal{H}\). Let \(\{\ket{x}\mid x\in\mathbb{F}_q\}\) be an orthonormal basis of \(\CC^q\), labelled by the elements of the finite field \(\FF_q\).

\begin{definition}
    For any \(a,b\in\FF_q\), let \(X(a)\) and \(Z(b)\) be the linear operators on \(\CC^q\) that are defined on the basis \(\{\ket{x}\mid x\in\FF_q\}\) by
    \begin{align*}
        X(a)\ket{x}=&\ket{x+a}\\
        Z(b)\ket{x}=&(-1)^{\tr(bx)}\ket{x}
    \end{align*}
    where \(\tr\) is the trace map \(\tr:\FF_q\to\FF_2:x\mapsto x+x^2+\cdots+x^{2^{h-1}}\).
    A \emph{Pauli operator} is a linear operator on \(\mathcal{H}\) of the form
    \[E=cX(a_1)Z(b_1)\otimes\cdots\otimes X(a_n)Z(b_n)\]
    for some \(a_1,\dots,a_n,b_1,\dots,b_n\in\FF_q\) and \(c\in\{\pm1,\pm i\}\). Its \emph{weight} is the number of nonidentity factors in its tensor product, i.e.\ the number of \(j\)'s for which \((a_j,b_j)\neq(0,0)\).
    The \emph{Pauli group} \(\P_n\) is the set of Pauli operators. %Note that \(\P_n=\P_1^{\otimes n}\).
\end{definition}

Note that 
\[Z(b)X(a)=(-1)^{\tr(ab)}X(a)Z(b).\]
In particular, \((\P_n,\cdot)\) is indeed a group, and every two elements commute or anti-commute. Pauli operators are unitary and their squares are either \(\id\) or \(-\id\) (where \(\id\) denotes the identity operator). %Observe that \(\P_n\) has size \(2^{2hn+2}\).

\begin{remark}\label{remark:pauli}
    There are other possible definitions of the (generalised) Pauli group, see \cite{nice error bases}. Ours is based upon a labelling with elements of a finite field, as introduced in \cite{ketkar}. Another common one uses a labelling with elements of the modular ring \(\ZZ/q\ZZ\). However, if \(q=2\), they coincide: \(X(1)\) and \(Z(1)\) are the ``ordinary'' Pauli operators, represented by the matrices \[X=\begin{pmatrix}0&1\\1&0\end{pmatrix}\text{\quad and \quad}Z=\begin{pmatrix}1&0\\0&-1\end{pmatrix}.\]
    For general \(q\), our results are specific to the definition of the Pauli operators based on the finite field.
\end{remark}

%\begin{definition}
 %   An \((\!(n,K,d)\!)_q\) \emph{quantum error-correcting code} is a \(K\)-dimensional subspace of \(\mathcal{H}\) that is able to detect\footnote{TODO i.e.\ for any orthogonal \(x\), \(\phi\) in this subspace, the images \(Ex\) and \(E'\phi\) are orthogonal for all errors \(E,E'\in\P_n\) of weight at most \(d-1\)

%    If \(K=q^k\), we use the notation \([\![n,k,d]\!]_q\) instead of \((\!(n,K,d)\!)_q\).
%    %We use the notation \([\![n,k,d]\!]_q\) for an \((\!(n,K,d)\!)_q\) quantum error-correcting code of dimension \(K=q^k\). %If the minimum distance is not specified, we use the notation \((\!(n,K)\!)_q\) or \([\![n,k]\!]_q\).
%\end{definition}

\subsection{Stabiliser codes}

Stabiliser codes were introduced in \cite{CRSS97} and \cite{gottesman96} independently.

\begin{definition}
    A \emph{stabiliser code} is a quantum error-correcting code of the form
    \[\Q=\left\{\ket{x}\in\mathcal{H}\mid E\ket{x}=\ket{x}\text{ for all }E\in\S\right\}\]
    for some subgroup \(\S\) of the Pauli group \(\P_n\). We assume that this subgroup is maximal with that property, i.e.\
    \[\S=\left\{E\in\P_n\mid E\ket{x}=\ket{x}\text{ for all }\ket{x}\in\Q\right\}.\]
    It is called the \emph{stabiliser} of \(\Q\). We also assume that \(\Q\) is nontrivial, i.e.\ \(\Q\neq\{0\}\), or equivalently, \(\S\neq\P_n\).
\end{definition}

%Assumptions throughout the paper?:
%\begin{itemize}
 %   \item Code is not empty. In particular, \(-\id\notin\S\).
  %  \item \(\mathrm{Centraliser}(\S)\) contains no elements of weight one
%\end{itemize}

\begin{lemma}[{\cite[Lemma~10]{ketkar}}]\label{lemma:abelian}
    \(\S\) is a stabiliser if and only if \(\S\) is an abelian subgroup of \(\P_n\) and \(-I\notin\S\).
\end{lemma}

\begin{lemma}\label{lemma:involutions}
    %Suppose that \(\Q\) has nonzero elements.
    %\begin{enumerate}[(i)]
        %\item \(\S\) is abelian.
        %\item
        Nonidentity elements of \(\S\) have order \(2\).%and thus, the elements of \(\S\) are hermitian.
    %\end{enumerate}
\end{lemma}
\begin{proof}
    Choose any \(E\in\S\) and let \(\ket{x}\) be a nonzero vector of \(\Q\). By definition, \(E\ket{x}=\ket{x}\).
    %\begin{enumerate}[(i)]
        %\item Since \(E_1,E_2\in\P_n\), either \(E_1E_2=E_2E_1\) or \(E_1E_2=-E_2E_1\). If \(E_1E_2=-E_2E_1\), then \(\ket{x}=E_1E_2\ket{x}=-E_2E_1\ket{x}=-\ket{x}\), a contradiction. So \(E_1\) and \(E_2\) commute.
        %\item
        Since \(E\in\P_n\), either \(E^2=\id\) or \(E^2=-\id\). If \(E^2=-\id\), then \(\ket{x}=E\ket{x}=E^2\ket{x}=-\ket{x}\), a contradiction. So \(E^2=\id\).
    %\end{enumerate}
\end{proof}

The following property is not as trivial as it looks. For a relatively short proof, see the first part of \cite[Theorem~5.4]{qeccandtheirgeom}.

\begin{lemma}[\cite{CRSS}]\label{lemma:dim}
    The dimension of \(\Q\) is \(q^n/|\S|\).
\end{lemma}

The \emph{normaliser} of \(\S\) in \(\P_n\) is
\(\N(\S)=\{E\in\P_n\mid E\S E^\dagger=\S\}\).
It can be shown to equal the centraliser, see \cite[Equation~(3.6)]{gottesman}.
Let \([\S]:=\{cE\mid c\in\{\pm1,\pm i\}\text{ and }E\in\S\}\). By \rref{Lemma}{lemma:abelian}, \([\S]\subseteq \N(\S)\). 
\begin{definition}\label{def:parameters}
    If \(|\S|=q^n\), we define the \emph{minimum distance \(d\)} of \(\Q\) as the minimum nonzero weight of the elements of \(\S\). If \(|\S|\neq q^n\), we define \(d\) as the minimum nonzero weight of the elements of 
    \(\N(\S)\setminus[\S]\).
    We define \(k:=\log_q(\dim\Q)\) and say that \(\Q\) is an \([\![n,k,d]\!]_q\) code.
\end{definition}

The motivation for this definition comes from the fact that a stabiliser code with minimum distance \(d\) can detect all errors up to weight \(d-1\). See \cite[Lemma 11]{ketkar}.

Note that \(|\P_n|=4\cdot q^{2n}=2^{2hn+2}\) and that \(\S\) is a subgroup of \(\P_n\), so the size of \(\S\) is a power of two. By \rref{Lemma}{lemma:dim}, \(|\S|=2^{h(n-k)}\), so \(kh\) must be an integer, though \(k\) need not be an integer.

\begin{definition}
    A code is \emph{pure} if \(d\) is the minimum nonzero weight of the elements of 
    \(\N(\S)\), or, equivalently, if the minimum nonzero weight of the elements of \(\S\) is at least \(d\).
\end{definition}

\begin{theorem}[Quantum Singleton bound {\cite[Theorem~2]{mds}}]\label{thm:mds}
    Let \(\Q\) be an \([\![n,k,d]\!]_q\) code. Then \[k\leq n-2d+2.\] If equality holds, the minimum nonzero weight of \(\S\) is \(n-d+2\). In particular, \(\Q\) is pure.
\end{theorem}

\subsection{Equivalent stabiliser codes}

\begin{definition}
    A \emph{Clifford operator} is a unitary operator on \(\mathcal{H}\) that normalises the Pauli group. In other words, it is an operator \(U\) for which
    \[UU^\dagger=\id\qquad\text{and}\qquad U\P_nU^\dagger=\P_n.\]
    The \emph{Clifford group \(\C_n\)} is the set of Clifford operators on \(\P_n\). A Clifford operator is called \emph{local} if it is the \(n\)-fold tensor product of elements of \(\C_1\). The \emph{local Clifford group \(\C_n^l\)} is the set of local Clifford operators, i.e.\ \(\C_n^l=\left(\C_1\right)^{\otimes n}\).
\end{definition}

Note that \(\P_n\subseteq\C_n^l\subseteq\C_n\).

\begin{definition}\label{def:equivalent}
    Two stabiliser codes are \emph{equivalent} if their stabilisers can be obtained from one another by a sequence of the following operations:
    \begin{enumerate}[(i)]
        \item A permutation on the positions of the tensor product.
        \item Multiplying elements by \(\pm1\) while preserving the group structure.
        \item Conjugation with a local Clifford operator, i.e.\ \(\S\mapsto U\S U^\dagger\) for some \(U\in\C_n^l\).
    \end{enumerate}
\end{definition}

We only allow multiplication of an element of \(\S\) by \(\pm1\) and not by \(\pm i\), since the latter would increase its order to \(4\), contradicting \rref{Lemma}{lemma:involutions}.

Note that a more general notion of equivalence exists, where the conjugation is not necessarily with a local Clifford operator, but rather with any unitary operator, see \cite{LU-LC}.

%Let \(w_i\) be the number of Pauli operators of \(\S\) of weight \(i\). We define the \emph{weight distribution} of \(\S\) to be \[(w_0,w_1,w_2,\ldots).\] \robin{this needed? not used further on I think}

\begin{lemma}\label{lemma:equivalent}
    Equivalent codes have the same parameters.% and the same weight distribution.
\end{lemma}
\begin{proof}
    The three operations in \rref{Definition}{def:equivalent} preserve the size of the stabiliser and the length and weight of Pauli operators.
\end{proof}

\begin{example}
    Note that conjugation with a \emph{nonlocal} Clifford operator does not necessarily preserve the weight. For example, if \(n=2\) and \(q=2\), then
    \[U=\frac{1}{2}\left(I\otimes I+X\otimes I+I\otimes Z-X\otimes Z\right)\]
    is a Clifford operator sending \(I\otimes X\) to \(X\otimes X\).
\end{example}

\subsection{Stabiliser codes as additive symplectic self-orthogonal codes}

Since \(\S\) is a subgroup of \(\P_n\) of size \(q^{n-k}\), it generated by a set \(\{E_1,\dots,E_r\}\) of \(r=h(n-k)\) Pauli operators.

\begin{definition}\label{def:tau}
Let \(\tau\) be the map
\begin{align*}
    \tau&:\P_n%/{\{\pm1,\pm i\}}
    \to\mathbb{F}_{2^h}^{2n}:\\
    &cX(a_1)Z(b_1)\otimes\cdots\otimes X(a_n)Z(b_n)\mapsto(a_1,\dots,a_n,b_1,\dots,b_n).
\end{align*}
\end{definition}
%We might also want to define
%\begin{align*}
%    \mathbf{XZ}&:\mathbb{F}_{2^h}^{2n}\to\P_n:\\
%    &(u_1,u_2,\dots,u_{2n})\mapsto X(u_1)Z(u_{n+1})\otimes\cdots\otimes X(u_n)Z(u_{2n})
%\end{align*}
%but this is trickier...

Using this map, \(\S\) can be represented by its \emph{stabiliser matrix}, an \(r\times2n\)-matrix whose rows are the vectors \(\tau(E_j)\) for each \(E_j\) in the generating set:
\[\left(\begin{array}{@{}cccc|cccc@{}}
a_{11}&a_{12}&\cdots&a_{1n}&b_{11}&b_{12}&\cdots&b_{1n}\\
a_{21}&a_{22}&\cdots&a_{2n}&b_{21}&b_{22}&\cdots&b_{2n}\\
\vdots&\vdots&&\vdots&\vdots&\vdots&&\vdots\\
a_{r1}&a_{r2}&\cdots&a_{rn}&b_{r1}&b_{r2}&\cdots&b_{rn}
\end{array}\right)\]
where \(E_j=c_jX(a_{j1})Z(b_{j1})\otimes\cdots\otimes X(a_{jn})Z(b_{jn})\). We disregard the phase factors \(c_j\), since stabilisers that differ in only these constants, are equivalent.

\begin{lemma}\label{lemma:tau}
    \(\tau\) is a group morphism from \((\P_n,\cdot)\) to \((\mathbb{F}_{2^h}^{2n},+)\), i.e.\ \[\tau(E_1E_2)=\tau(E_1)+\tau(E_2).\]
\end{lemma}
\begin{proof}
    Since \(X(a)X(b)=X(a+b)\) and \(Z(a)Z(b)=Z(a+b)\), together with the fact that Pauli operators either commute or anti-commute (and \(\tau\) ignores the factor \(\pm1\)).
\end{proof}

The fact that the elements of \(\S\) commute, translates into the property that the rows of this matrix are orthogonal with respect to the \emph{trace-symplectic inner product}
\[\langle(a|b),(a'|b')\rangle_s=\tr(a\cdot b'+a'\cdot b)=\sum_{i=1}^n\tr(a_ib'_i+a'_ib_i).\]

\begin{lemma}[\cite{CRSS}]\label{lemma:commute}
    Let \(E_1,E_2\in\P_n\). Then \(E_1E_2=E_2E_1\Longleftrightarrow\langle\tau(E_1),\tau(E_2)\rangle_s=0\).
\end{lemma}
\begin{proof}
    Denote \(E_1=X(a_1)Z(b_2)\otimes\cdots\otimes X(a_n)Z(b_n)\) and \(E_2=X(a'_1)Z(b'_2)\otimes\cdots\otimes X(a'_n)Z(b'_n)\). Since \[Z(b)X(a)=(-1)^{\tr(ab)}X(a)Z(b),\] we have
    \[E_1E_2=(-1)^{\sum_{i=1}^n\tr(a'_ib_i)}X(a_1+a'_1)Z(b_1+b'_1)\otimes\cdots\otimes X(a_n+a'_n)Z(b_n+b'_n)\]
    and
    \[E_2E_1=(-1)^{\sum_{i=1}^n\tr(a_ib'_i)}X(a_1+a'_1)Z(b_1+b'_1)\otimes\cdots\otimes X(a_n+a'_n)Z(b_n+b'_n)\]
    so
    \[E_1E_2=(-1)^{\langle(a|b),(a'|b')\rangle_s}E_2E_1.\]
\end{proof}

In other words, the stabiliser matrix is a generator matrix of a code \(C\) that is contained in its \emph{symplectic dual}
\[C^{\perp_s}=\{(a|b)\in\mathbb{F}_q^{2n}\mid\langle(a|b),(a'|b')\rangle_s=0\text{ for all }(a'|b')\in C\}.\]

The \emph{symplectic weight} of a vector \((a|b)\in\mathbb{F}_q^{2n}\) is defined as
\[\mathrm{swt}(a|b)=\left|\left\{j\in\{1,\dots,n\}\mid(a_j,b_j)\neq(0,0)\right\}\right|.\]
So the weight of a Pauli operator \(E\in\P_n\) is equal to \(\mathrm{swt}(\tau(E))\).
The symplectic weight of a set of vectors is defined as the minimum among those weights, not including the zero weight.

\begin{theorem}[{\cite[Theorem~13]{ketkar}}]\label{thm:ketkar}
    Let \(\Q\) be an \([\![n,k,d]\!]_q\) code with stabiliser \(\S\). Then \(C=\tau(\S)\) is an additive code \(C\subseteq\FF_q^{2n}\) of size \(|C|=q^{n-k}\) such that \(C\subseteq C^{\perp_s}\) for which \(\mathrm{swt}(C^{\perp_s})=d\) if \(k=0\) and \(\mathrm{swt}(C^{\perp_s}\setminus C)=d\) if \(k>0\). Conversely, given such a code \(C\), there is an \([\![n,k,d]\!]_q\) code with stabiliser \(\S\) such that \(C=\tau(\S)\).
\end{theorem}

We should mention that the inverse operation is a bit trickier. Given a vector \((a|b)\in\FF_q^{2n}\), there are four Pauli operators \(E\) such that \(\tau(E)=(a|b)\). They differ in a phase factor. However, this phase factor cannot be chosen arbitrarily. For example, if \(n=1\) and \(q=2\), the vector \((1,1)\) should not be mapped to \(XZ\) but rather to \(iXZ\) or \(-iXZ\) since \(XZ\) has order \(4\) instead of the required order \(2\) (see \rref{Lemma}{lemma:involutions}). For each vector of the code, we have two options for the phase factor. However, up to equivalence, there is only one pre-image of \(\tau\). We will therefore use the notation \(\tau^{-1}(C)\) for the \emph{class} of all stabilisers whose image is \(C\). In other words,
\[\tau^{-1}(C)=[\S]\qquad\text{where }\tau(\S)=C.\]
%But with some {\color{red}abuse of notation}, we will denote any stabiliser \(\S\) for which \(C=\tau_h(\S)\) as \(\tau_h^{-1}(C)\) since they are all equivalent anyway. From now on, we will often say stabiliser code when we actually mean these equivalent codes.

%\subsection{Equivalent stabiliser codes, revisited}

%\begin{theorem}[{\cite{GGMG}}]
%    Binary stabiliser codes are equivalent if and only if the corresponding symplectic self-orthogonal codes are equivalent.
%\end{theorem}
%\begin{proof}
%    Conjugation with a local Clifford operator in one position of the tensor product, comes down to a permutation of the ``symbols'' \((1,0)\), \((0,1)\) and \((1,1)\). (There are only 6 different \(2\times2\) symplectic matrices over \(\mathbb{F}_2\).) TODO explain better.
%\end{proof}

\section{Stabiliser codes over fields of even order as binary stabiliser codes}\label{sec:main}

In this section, we work with multiple values of \(n\) and \(h\). We will therefore denote the Hilbert space \((\CC^{2^h})^{\otimes n}\) by \(\mathcal{H}_{n,h}\), the corresponding Pauli group by \(\P_{n,h}\), its Clifford group by \(\C_{n,h}\) and the local Clifford group by \(\C_{n,h}^l\). The map \(\tau\) becomes \(\tau_{n,h}\).

From now on, we also use \(i\) as an index, and not necessarily the complex unit.

\subsection{The trace-orthogonal trick}

Stabiliser codes over fields of even order are connected to binary stabiliser codes by Theorem~\ref{thm:main} below. It is a reformulation of \cite[Lemma~76]{ketkar}, a result that was already implicitly contained in \cite{ashikminknill}. For the sake of completeness, we include proofs of this theorem and its preparatory lemmas. To this end, we use a so-called trace-orthogonal basis, which was introduced in \cite{ashikminknill}. See for example also \cite{traceortho}.

A basis \(\{e_1,\dots,e_h\}\) of \(\mathbb{F}_{2^h}\cong\mathbb{F}_2^h\) over \(\mathbb{F}_2\) is called \emph{trace-orthogonal} if it has the property that \(\mathrm{tr}(e_ie_j)=\delta_{ij}\) for all \(i,j\in\{1,\dots,h\}\). Note that such a basis always exists \cite[Theorem~4]{selfdual}.

Recall that \(\{\ket{x}\mid x\in\FF_{2^h}\}\) is an orthonormal basis of \(\CC^{2^h}\), labelled by the elements of the finite field \(\FF_{2^h}\). This allows us to translate the isomorphism \(\FF_{2^h}\cong\FF_2^h\) into a bijective linear map between \(\mathcal{H}_{n,h}\) and \(\mathcal{H}_{hn,1}\).

%Since \(\{e_1,\dots,e_h\}\) is a trace-orthogonal basis, we have the following property.
%\begin{lemma}\label{lemma:inproduct}
%    The map \(\phi_{n,h}\) preserves the trace-symplectic inner product, i.e.\ \[\langle u,v\rangle_s=\langle\phi_{n,h}(u),\phi_{n,h}(v)\rangle_s.\]
%\end{lemma}
%\begin{proof}
%    Let \(u=(u_1,\dots,u_{2n})\) %=(u_{11}e_1+\cdots+u_{1h}e_h,\dots,u_{2n,1}e_1+\cdots+u_{2n,h}e_h)\) and \(v=(v_1,\dots,v_{2n})\) %=(v_{11}e_1+\cdots+v_{1h}e_h,\dots,v_{2n,1}e_1+\cdots+v_{2n,h}e_h)\) be two vectors of \((\mathbb{F}_{2^h})^{2n}\), and write their entries as \(u_i=\sum_{j=1}^hu_{ij}e_{j}\) resp.\ \(v_i=\sum_{j=1}^hv_{ij}e_{j}\) (\(i=1,\dots,2n\)). Then
%    \begin{align*}
%        \langle u,v\rangle_s&=\sum_{i=1}^n\mathrm{tr}\left(u_iv_{n+i}+v_iu_{n+i}\right)\\
        %&=\sum_{i=1}^n\sum_{j,l=1}^h\mathrm{tr}\left(u_{i,j}v_{n+i,l}e_je_l-v_{i,l}u_{n+i,j}e_je_l\right)\\
        %&=\sum_{i=1}^n\sum_{j,l=1}^h\left(u_{i,j}v_{n+i,l}+v_{i,l}u_{n+i,j}\right)\mathrm{tr}(e_je_l)\\
        %&=\sum_{i=1}^n\sum_{j=1}^h\left(u_{i,j}v_{n+i,j}+v_{i,j}u_{n+i,j}\right)\\
        %&=\langle\phi_{n,h}(u),\phi_{n,h}(v)\rangle_s.
    %\end{align*}
%\end{proof}

\begin{definition}
    Let \(\{e_1,\dots,e_h\}\) be a trace-orthogonal basis of \(\mathbb{F}_{2^h}\) over \(\mathbb{F}_2\). Let \(\Psi_{n,h}\) be the linear transformation \(\mathcal{H}_{n,h}\to\mathcal{H}_{hn,1}\) that is defined on the basis \(\{\ket{x}\mid x\in\FF_{2^h}\}^{\otimes n}\) by
    %\[\Phi_h\ket{u_1e_1+\cdots+u_he_h}=\ket{u_1}\otimes\cdots\otimes\ket{u_h}\]
    %\begin{align*}\eta_h&:\mathcal{H}_{1,h}\to\mathcal{H}_{h,1}:\\&\ket{u_1e_1+\cdots+u_he_h}\mapsto \ket{u_1}\otimes\cdots\otimes\ket{u_h}\end{align*}
    %and define \(\Phi_{n,h}:=\Phi_h^{\otimes n}\).
    %In other words, \(\Phi_{n,h}\) is the linear operator \(\mathcal{H}_{n,h}\to\mathcal{H}_{hn,1}\) for which
    \begin{multline*}\Psi_{n,h}\ket{x_{11}e_1+\cdots+x_{1h}e_h}\otimes\cdots\otimes \ket{x_{n1}e_1+\cdots+x_{nh}e_h}\\= \ket{x_{11}}\otimes\cdots\otimes \ket{x_{1h}}\otimes\cdots\otimes \ket{x_{n1}}\otimes\cdots\otimes \ket{x_{nh}}.\end{multline*}

    For an operator \(A\) on \(\mathcal{H}_{n,h}\), define \[\psi_{n,h}(A):=\Psi_{n,h}A\Psi_{n,h}^{-1}\] to be the corresponding operator on \(\mathcal{H}_{hn,1}\).
\end{definition}

%\begin{align*}
    %\Phi_{n,h}&:\mathcal{H}_{n,h}\to\mathcal{H}_{hn,1}:\\
    %&\ket{u_{11}e_1+\cdots+u_{1h}e_h}\otimes\cdots\otimes \ket{u_{n1}e_1+\cdots+u_{nh}e_h}\\
    %&\mapsto \ket{u_{11}}\otimes\cdots\otimes \ket{u_{1h}}\otimes\cdots\otimes \ket{u_{n1}}\otimes\cdots\otimes \ket{u_{nh}}.
%\end{align*}

%With some abuse of notation, one could say that \(\Phi_{n,h}\ket{x}=\ket{\phi_{n,h}(x)}\).

%Note that \(\Phi_{n,h}\) is a unitary transformation, as it maps an orthonormal basis to an orthonormal basis.
%\begin{lemma}
%    \(\Phi_{n,h}\) is a unitary transformation. %beware that it is an operator between different spaces
%\end{lemma}
%\begin{proof}
%    It maps an orthonormal basis to an orthonormal basis.
%\end{proof}

\begin{lemma}\label{lemma:rhomultiplicative}
    \(\psi_{n,h}(cAB)=c\psi_{n,h}(A)\psi_{n,h}(B)\) and \(\psi_{n,h}(A^\dagger)=\psi_{n,h}(A)^\dagger\).
\end{lemma}
\begin{proof}
    The first equality is true by definition. The second one follows from the fact that \(\Psi_{n,h}\) is a unitary transformation (it maps an orthonormal basis to an orthonormal basis). %maps an orthonormal basis to another orthonormal basis.
\end{proof}

\begin{lemma}\label{lemma:rhopauli}
    \(\psi_{n,h}(\P_{n,h})=\P_{hn,1}\). More precisely, \(\psi_{n,h}\) maps \begin{multline*}cX(a_{11}e_1+\cdots+a_{1h}e_h)Z(b_{11}e_1+\cdots+b_{1h}e_h)\otimes\cdots\\\cdots\otimes X(a_{n1}e_1+\cdots+a_{nh}e_h)Z(b_{n1}e_1+\cdots+b_{nh}e_h)\end{multline*} to \[cX(a_{11})Z(b_{11})\otimes\cdots\otimes X(a_{1h})Z(b_{1h})\otimes\cdots\otimes X(a_{n1})Z(b_{n1})\otimes\cdots\otimes X(a_{nh})Z(b_{nh}).\]
\end{lemma}
\begin{proof}
    It suffices to prove this for \(n=1\) because \(\psi_{n,h}=\psi_{1,h}^{\otimes n}\) and \(\P_{n,h}=\P_{1,h}^{\otimes n}\). %In other words, we must prove that the restriction of \(\phi_{1,h}\) to the Pauli group \(\P_{1,h}\) is equal to
%\begin{align*}
%    \phi_{1,h}&:\P_{1,h}\to\P_{h,1}:\\
%    &cX(u_1e_1+\cdots+u_he_h)Z(v_1e_1+\cdots+v_he_h)\mapsto cX(u_1)Z(v_1)\otimes\cdots\otimes X(u_h)Z(v_h).
%\end{align*}
    Moreover, since \(X(a+b)=X(a)X(b)\) and \(Z(a+b)=Z(a)Z(b)\), and since \(\psi_{n,h}\) is product-preserving (\rref{Lemma}{lemma:rhomultiplicative}), it is enough to prove that
    \(\psi_{1,h}(X(e_i))=X_i\) and \(\psi_{1,h}(Z(e_i))=Z_i\), where \(X_i\) resp.\ \(Z_i\) denotes the weight one Pauli operator with an \(X\) resp.\ \(Z\) in the \(i\)-th position. %, i.e.\
    %\[X_i=I\otimes\cdots\otimes\underbrace{X}_\text{\(i\)-th position}\otimes\cdots\otimes I\]
    %\[Z_i=I\otimes\cdots\otimes\underbrace{Z}_\text{\(i\)-th position}\otimes\cdots\otimes I\]
    Let \(x_1,\dots,x_h\in\FF_2\). Then \begin{align*}
        \psi_{1,h}(X(e_i))\ket{x_1}\otimes\cdots\otimes\ket{x_h}&=\Psi_{1,h} X(e_i)\ket{x_1e_1+\cdots+x_he_h}\\
        &=\Psi_{1,h}\ket{x_1e_1+\cdots+(x_i+1)e_i+\cdots+x_he_h}\\
        &=\ket{x_1}\otimes\cdots\otimes\ket{x_i+1}\otimes\cdots\otimes\ket{x_h}\\
        &=X_i\ket{x_1}\otimes\cdots\otimes\ket{x_h}
    \end{align*}
    and
    \begin{align*}
        \psi_{1,h}(Z(e_i))\ket{x_1}\otimes\cdots\otimes\ket{x_h}&=\Psi_{1,h}Z(e_i)\ket{x_1e_1+\cdots+x_he_h}\\
        &=\Psi_{1,h}(-1)^{\tr(e_i\cdot(x_1e_1+\cdots+x_he_h))}\ket{x_1e_1+\cdots+x_he_h}\\
        &=(-1)^{x_i}\ket{x_1}\otimes\cdots\otimes\ket{x_h}\\
        &=Z_i\ket{x_1}\otimes\cdots\otimes\ket{x_h}
    \end{align*}
    where we used the fact that \(\{e_1,\dots,e_h\}\) is a trace-orthogonal basis.
\end{proof}

%Lemma~\ref{lemma:commute} and Lemma~\ref{lemma:inproduct} together give rise to the following:

%\begin{theorem}
%    \(\phi_h\left(\mathrm{Centraliser}(\S)\right)=\mathrm{Centraliser}(\phi_h(\S))\).
%\end{theorem}

The following theorem is a restatement of \cite[Lemma~76]{ketkar}, but was already implicitly contained in \cite{ashikminknill} and was extended to arbitrary characteristic in \cite{gross}, see also \cite{heinrich}. We include a short proof for the sake of completeness.

\begin{theorem}[{\cite[Lemma~76]{ketkar}}]\label{thm:main}
    \(\Psi_{n,h}\) induces a bijection between \([\![n,k,d]\!]_{2^h}\) codes and \([\![hn,hk,d']\!]_2\) codes and \(\psi_{n,h}\) induces a bijection between their respective stabilisers. Moreover, \(d\leq d'\leq hd\).
\end{theorem}

%\begin{center}\begin{tikzcd}
%    \Q \arrow[d] &[20mm] \Q' \arrow[d]\\\S \arrow[u] \arrow[d, "\tau_{n,h}"] \arrow[r, "\phi_{n,h}"] & \S' \arrow[u] \arrow[d, "\tau_{hn,1}"]\\C \arrow[r, "\phi_{n,h}"] & C'
%\end{tikzcd}\end{center}

\begin{proof}
    \(\Q\) is an \([\![n,k,d]\!]_{2^h}\) code with stabiliser \(\S\) for some \(d\) if and only if
    \begin{align*}
        \Q&=\left\{\ket{x}\in\mathcal{H}_{n,h}\mid E\ket{x}=\ket{x}\text{ for all }E\in\S\right\}\\
        \S&=\left\{E\in\P_{n,h}\mid E\ket{x}=\ket{x}\text{ for all }\ket{x}\in\Q\right\}
    \end{align*}
    and \(\dim\Q=(2^h)^k\). Since \(\Psi_{n,h}\) is a bijection, and by \rref{Lemma}{lemma:rhomultiplicative} and \rref{Lemma}{lemma:rhopauli}, this is equivalent to saying that
    \begin{align*}
        \Psi_{n,h}(\Q)%&=\left\{\Phi_{n,h}\ket{x}\mid\ket{x}\in\mathcal{H}_{n,h}\text{ and }\phi_{n,h}(E)\Phi_{n,h}\ket{x}=\Phi_{n,h}\ket{x}\text{ for all }E\in\S\right\}\\
        &=\left\{\ket{x}\in\mathcal{H}_{hn,1}\mid E\ket{x}=\ket{x}\text{ for all }E\in\psi_{n,h}(\S)\right\}\\
        \psi_{n,h}(\S)%&=\left\{\phi_{n,h}(E)\mid E\in\P_{n,h}\text{ and }\phi_{n,h}(E)\Phi_{n,h}\ket{x}=\Phi_{n,h}\ket{x}\text{ for all }\ket{x}\in\Q\right\}\\
        &=\left\{E\in\P_{hn,1}\mid E\ket{x}=\ket{x}\text{ for all }\ket{x}\in\Psi_{n,h}(\Q)\right\}
    \end{align*}
    where \(\dim\Psi_{n,h}(\Q)=\dim\Q=2^{hk}\). In other words, \(\Q':=\Psi_{n,h}(\Q)\) is an \([\![hn,hk,d']\!]_2\) code with stabiliser \(\S':=\psi_{n,h}(\S)\).

    \(d\leq d'\). There is an element of weight \(d'\) in \(\N(\S')\setminus\S'\) (or \(\S'\) if \(k=0\)). From \rref{Lemma}{lemma:rhopauli}, it follows that applying \(\Psi_{n,h}^{-1}\) yields an element of weight at most \(d'\) in \(\N(\S)\setminus\S\) (resp.\ \(\S\)).
    
    \(d'\leq hd\). There is an element of weight \(d\) in \(\N(\S)\setminus\S\) (or \(\S\) if \(k=0\)). \rref{Lemma}{lemma:rhopauli} tells us that applying \(\Psi_{n,h}\) results in an element of weight at most \(hd\) in \(\N(\S')\setminus\S'\) (resp.\ \(\S'\)).
\end{proof}

Though \(\Psi_{n,h}\) induces a bijection, it does not preserve equivalence. See \rref{Section}{sec:hequiv} below.

\begin{definition}
Let \(\{e_1,\dots,e_h\}\) be a trace-orthogonal basis of \(\mathbb{F}_{2^h}\cong\mathbb{F}_2^h\) over \(\mathbb{F}_2\). Define
\begin{align*}
    \phi_{n,h}&:(\mathbb{F}_{2^h})^{2n}\to(\mathbb{F}_2)^{2hn}:\\
    &    (x_{11}e_1+\cdots+x_{1h}e_h,\dots,x_{2n,1}e_1+\cdots+x_{2n,h}e_h)\\&\mapsto(x_{11},\dots,x_{1h},\dots,x_{2n,1},\dots,x_{2n,h}).
\end{align*}
\end{definition}

\begin{corollary}\label{cor:commutativediagram}
    %\(\tau_{hn,1}\circ\phi_{n,h}=\phi_{n,h}\circ\tau_{n,h}\). I.e.\ 
    The following diagram commutes:
    \begin{center}\begin{tikzcd}
    \P_{n,h} \arrow[d, "\tau_{n,h}"] \arrow[r, "\psi_{n,h}"] &[20mm] \P_{hn,1} \arrow[d, "\tau_{hn,1}"]\\
    \FF_{2^h}^{2n} \arrow[r, "\phi_{n,h}"] & \FF_2^{2hn}
\end{tikzcd}\end{center}
\end{corollary}
\begin{proof}
    This follows from \rref{Lemma}{lemma:rhopauli} and from the definitions.
\end{proof}

\begin{corollary}
    \(\phi_{n,h}\) induces a bijection between symplectic self-orthogonal additive codes over \(\FF_{2^h}\) of length \(2n\) and size \(2^{h(n-k)}\) and binary symplectic self-orthogonal additive codes of length \(2hn\) and size \(2^{h(n-k)}\). In particular, \(\phi_{n,h}\) preserves the trace-symplectic inner product.
\end{corollary}
\begin{proof}
    \(C\) is a symplectic self-orthogonal additive code \(C\subseteq\FF_{2^h}^{2n}\) of size \(2^{h(n-k)}\) if and only if \(\S:=\tau_{n,h}^{-1}(C)\)
    is the stabiliser of an \([\![n,k,d]\!]_{2^h}\) code, by \rref{Theorem}{thm:ketkar}. By \rref{Theorem}{thm:main}, this is equivalent to \(\S':=\psi_{n,h}(\S)\) being the stabiliser of an \([\![hn,hk,d']\!]_2\) code. We can now apply \rref{Theorem}{thm:ketkar} again to see that, equivalently, \(C':=\tau_{hn,1}(\S')\) is a symplectic self-orthogonal additive code \(C'\subseteq\FF_2^{2hn}\) of size \(2^{h(n-k)}\). Finally, note that \(C'=\phi_{n,h}(C)\) by \rref{Corollary}{cor:commutativediagram}.
\end{proof}

Alternatively, we could have proved the above corollary by showing that \(\phi_{n,h}\) preserves the trace-symplectic inner product (which follows from the trace-orthogonality), and then deducing \rref{Theorem}{thm:main} from it by using \rref{Corollary}{cor:commutativediagram}.

The above theorem provides a construction for new quantum codes from old ones. Applying \(\Psi_{n,h}\) corresponds to the following operation. We start from an \([\![n,k,d]\!]_{2^h}\) code and split the columns of the corresponding symplectic self-orthogonal code according to a trace orthonormal basis to obtain an \([\![hn,hk,d']\!]_2\) code.

\begin{example}
    Let \(\alpha\in\FF_4\setminus\{0,1\}\). Then \(\alpha^2=\alpha+1\) and \(\{\alpha,\alpha^2\}\) is a trace-orthogonal basis  of \(\FF_4\) over \(\FF_2\). Consider the \([\![2,1,2]\!]_4\) code with stabiliser matrix
    \[\left(\begin{array}{@{}cc|cc@{}}
    1&1&\alpha&0\\
    0&1&1&1
    \end{array}\right).\]
    Applying \(\Psi_{2,2}\) on this code, or, in other words, applying \(\phi_{2,2}\) on the rows of this matrix, results in the \([\![4,2,2]\!]_2\) code with stabiliser matrix
    \[\left(\begin{array}{@{}cc;{2pt/1pt}cc|cc;{2pt/1pt}cc@{}}
    1&1&1&1&1&0&0&0\\
    0&0&1&1&1&1&1&1
    \end{array}\right).\]
\end{example}

We can also apply \(\Psi_{n,h}^{-1}\): we start from an \([\![hn,hk,d']\!]_2\) code and partition the columns into sets of \(h\) columns which we merge together to obtain an \([\![n,k,d]\!]_{2^h}\) code.

\begin{example}
    Consider the \([\![12,6,3]\!]_2\) code with stabiliser matrix
    \[\left(\begin{array}{@{}cccccccccccc|cccccccccccc@{}}
    1&0&0&0&0&0&0&1&0&1&0&1&0&0&1&0&1&1&1&0&0&0&0&0\\
    0&1&0&0&0&0&1&0&1&0&1&0&1&0&0&1&1&1&0&1&0&1&0&0\\
    0&0&1&0&0&0&0&1&0&1&1&0&0&1&0&0&1&0&0&0&1&0&0&1\\
    0&0&0&1&0&0&0&1&1&0&0&0&1&0&1&0&1&0&1&1&0&0&1&0\\
    0&0&0&0&1&0&1&1&0&0&0&1&1&1&1&1&0&1&0&0&0&0&1&1\\
    0&0&0&0&0&1&0&0&0&0&1&0&0&0&1&1&1&0&1&1&1&1&1&1
    \end{array}\right).\]
    
    If we merge every two columns with \(\phi_{6,2}^{-1}\) according to the basis \(\{\alpha,\alpha^2\}\) from before, we get the stabiliser matrix
    \[\left(\begin{array}{@{}cccccc|cccccc@{}}
    \alpha&0&0&\alpha^2&\alpha^2&\alpha^2&0&\alpha&1&\alpha&0&0\\
    \alpha^2&0&0&\alpha&\alpha&\alpha&\alpha&\alpha^2&1&\alpha^2&\alpha^2&0\\
    0&\alpha&0&\alpha^2&\alpha^2&\alpha&\alpha^2&0&\alpha&0&\alpha&\alpha^2\\
    0&\alpha^2&0&\alpha^2&\alpha&0&\alpha&\alpha&\alpha&1&0&\alpha\\
    0&0&\alpha&1&0&\alpha^2&1&1&\alpha^2&0&0&1\\
    0&0&\alpha^2&0&0&\alpha&0&1&\alpha&1&1&1
    \end{array}\right)\]
    of a \([\![6,3,2]\!]_4\) code.
    Let \(\beta\in\FF_8\) such that \(\beta^3=\beta^2+1\). Then \(\{\beta,\beta^2,\beta^4\}\) is a trace-orthogonal basis of \(\FF_8\) over \(\FF_2\). If we merge every three columns according to this basis, we get the stabiliser matrix
    \[\left(\begin{array}{@{}cccc|cccc@{}}
    \beta&0&\beta^2&\beta^3&\beta^4&\beta^5&\beta&0\\
    \beta^2&0&\beta^3&\beta^2&\beta&1&\beta^2&\beta\\
    \beta^4&0&\beta^2&\beta^6&\beta^2&\beta^2&\beta^4&\beta^4\\
    0&\beta&\beta^5&0&\beta^3&\beta^2&\beta^6&\beta^2\\
    0&\beta^2&\beta^6&\beta^4&1&\beta^3&0&\beta^5\\
    0&\beta^4&0&\beta^2&\beta^4&\beta^6&1&1
    \end{array}\right).\]
    of a \([\![4,2,2]\!]_8\) code.
\end{example}

In a similar fashion, we can construct a \([\![6,2,3]\!]_4\) code and a \([\![3,1,2]\!]_{16}\) code from a \([\![12,4,4]\!]_2\) code. One can also construct a \([\![7,1,3]\!]_4\) code from a \([\![14,2,5]\!]_2\) code and a \([\![8,0,4]\!]_4\) code from a \([\![16,0,6]\!]_2\) code.

Not all partitions of the columns into \(h\)-sets are interesting however. We are only interested in those \(h\)-sets that contain linearly independent columns. Otherwise, \(\N(\S)\) contains elements of weight one. See \rref{Lemma}{lemma:generalposition} below.

The minimum distance can be calculated exactly using \rref{Definition}{def:parameters} or via \rref{Theorem}{thm:ketkar}. \rref{Theorem}{thm:main} already indicates which values of the minimum distance we might expect.

We can also combine both maps \(\Psi_{n,h}\) and \(\Psi_{n',h'}^{-1}\) for possibly different \(n,h,n'\) and \(h'\).

\begin{corollary}
    Let \(\pi\) denote a permutation on \(hn\) elements. Then \(\Psi_{n',h'}^{-1}\circ\pi\circ\Psi_{n,h}\) maps an \([\![n,k,d]\!]_{2^h}\) code to an \([\![n',k',d']\!]_{2^{h'}}\) code, where $h'n'=hn$ and \(h'k'=hk\).
\end{corollary}
\begin{proof}
    This follows from \rref{Theorem}{thm:main}, and the fact that a permutation of tensor positions gives an equivalent code.
\end{proof}

%We should check some examples of equivalent ququat codes giving nonequivalent binary codes.

\subsection{Equivalence of stabiliser codes over fields of even order}\label{sec:hequiv}% as binary codes}

%In this paragraph, we still assume that \(\{e_1,\dots,e_h\}\) is a basis of \(\mathbb{F}_{2^h}\cong\mathbb{F}_2^h\) over \(\mathbb{F}_2\) with the property that \(\mathrm{tr}(e_ie_j)=\delta_{ij}\) for all \(i,j\in\{1,\dots,h\}\). Recall that such a basis always exists \cite[Theorem~4]{selfdual}.

%\begin{definition}
%    Define the linear map
%    \[\phi:\mathbb{C}^{2^h}\to\left(\mathbb{C}^2\right)^{\otimes h}:\quad\ket{\sum a_ie_i}\mapsto\bigotimes\ket{a_i}.\]
%    We will use the notation \(A^\phi=\phi\circ A\circ\phi^{-1}\) for its action on the dual space of \(\mathbb{C}^{2^h}\).
%\end{definition}

Although \(\Psi_{n,h}\) is a bijection between \([\![n,k,d]\!]_{2^h}\) codes and \([\![hn,hk,d']\!]_2\) codes, it does not necessarily preserve the equivalence of codes. %For example, a conjugation with a local Clifford operator in the \([\![n,k,d]\!]_{2^h}\) code may correspond to a conjugation with a \emph{nonlocal} Clifford operator in the \([\![hn,hk,d']\!]_2\) code (see Theorem~\ref{thm:hequiv} below).
In this section, we establish what this equivalence translates to in the underlying binary code, and in the corresponding symplectic self-orthogonal code.
%Let \(X_i\) resp.\ \(Z_i\) denote the weight one Pauli operator with an \(X\) resp.\ \(Z\) in the \(i\)-th position:
%    \[X_i=I\otimes\cdots\otimes\underbrace{X}_\text{\(i\)-th position}\otimes\cdots\otimes I\]
%    \[Z_i=I\otimes\cdots\otimes\underbrace{Z}_\text{\(i\)-th position}\otimes\cdots\otimes I\]

%\begin{lemma}
%    \begin{enumerate}[(i)]
%        \item \(X(e_i)^\phi=X_i\)
%        \item \(Z(e_i)^\phi=Z_i\)
%        \item \(\left(\P_{1,h}\right)^\phi=\P_{h,1}\)
%        \item \(\left(\C_{1,h}\right)^\phi=\C_{h,1}\)
%        \item \(\left(\C_{n,h}^l\right)^{\phi^{\otimes n}}=\left(\C_{h,1}\right)^{\otimes n}\)
%    \end{enumerate}
%\end{lemma}
%The map \(\phi_h\) naturally extends to a linear operator \(\Phi_h:\CC^{2^h}\to\left(\CC^2\right)^{\otimes h}\) that is defined on the basis \(\{\ket{x}\mid x\in\FF_{2^h}\}\) of \(\CC^{2^h}\) by
%\[\Phi_h\ket{u_1e_1+\cdots+u_he_h}=\ket{u_1}\otimes\cdots\otimes\ket{u_h}.\]

%And similarly, \(\phi_{n,h}\) TODO not really extends to the operator
%\begin{align*}
%    \Phi_{n,h}&:\mathcal{H}_{n,h}\to\mathcal{H}_{hn,1}\\
    %&\ket{u_{11}e_1+\cdots+u_{1h}e_h}\otimes\cdots\otimes\ket{u_{n1}e_1+\cdots+u_{nh}e_h}\\
    %&\mapsto\ket{u_{11}}\otimes\cdots\otimes\ket{u_{1h}}\otimes\cdots\otimes\ket{u_{n1}}\otimes\cdots\otimes\ket{u_{nh}}.
%\end{align*}
We start by showing that a local Clifford operator on \(\P_{h,n}\) is equivalent to the tensor product of \(n\) (global) Clifford operators on \(\P_{h,1}\).

\begin{lemma}\label{lemma:clifford}
    \(\psi_{n,h}\left(\C_{n,h}^l\right)=\C_{h,1}^{\otimes n}\).
\end{lemma}
\begin{proof}
    It suffices to show that \(\psi_{1,h}\left(\C_{1,h}\right)=\C_{h,1}\), since \(\psi_{n,h}=\psi_{1,h}^{\otimes n}\) and \(\C_{n,h}^l=\C_{1,h}^{\otimes n}\).
    Let \(U\) be an operator on \(\mathcal{H}_{1,h}\). Then
    \begin{align*}
        U\in\C_{1,h}&\Leftrightarrow UU^\dagger=\id\text{ and }U\P_{1,h}U^\dagger=\P_{1,h}\\
        &\Leftrightarrow \psi_{1,h}(U)\psi_{1,h}(U)^\dagger=\id\text{ and }\psi_{1,h}(U)\P_{h,1}\psi_{1,h}(U)^\dagger=\P_{h,1}\\
        &\Leftrightarrow\psi_{1,h}(U)\in\C_{h,1}
    \end{align*}
    by \rref{Lemma}{lemma:rhomultiplicative} and \rref{Lemma}{lemma:rhopauli}.
\end{proof}

\begin{theorem}\label{thm:hequiv}
    Two \([\![n,k,d]\!]_{2^h}\) codes are equivalent if and only if the stabilisers of the corresponding \([\![hn,hk,d']\!]_2\) codes, obtained by \(\psi_{n,h}\), can be converted into one another by a sequence of the following operations:
    \begin{enumerate}[(i)]
        \item A permutation on the \(n\) groups of \(h\)-fold tensors in the tensor product.
        \item Multiplying elements by \(\pm1\) while preserving the group structure.
        \item Conjugation with a Clifford operator of \(\left(\C_{h,1}\right)^{\otimes n}\).
    \end{enumerate}
\end{theorem}
\begin{proof}
    We show that the operations (i)--(iii) in \rref{Definition}{def:equivalent} for equivalent \([\![n,k,d]\!]_{2^h}\) codes correspond to the operations stated above for \([\![hn,hk,d']\!]_2\) codes.
    \begin{enumerate}[(i)]
        \item This follows from \rref{Lemma}{lemma:rhopauli}.
        \item By linearity of \(\psi_{n,h}\) (\rref{Lemma}{lemma:rhomultiplicative}).
        \item Suppose that \(\S,\S'\leqslant\P_{n,h}\) are stabilisers. Then
        \begin{align*}
            &\S'=U\S U^\dagger\text{ for some }U\in\C_{n,h}^l\\
            &\Leftrightarrow\psi_{n,h}(\S')=\psi_{n,h}(U)\psi_{n,h}(\S)\psi_{n,h}(U)^\dagger\text{ for some }U\in\C_{n,h}^l\\
            &\Leftrightarrow\psi_{n,h}(\S')=U\psi_{n,h}(\S)U^\dagger\text{ for some }U\in\C_{h,1}^{\otimes n}
        \end{align*}
        by \rref{Lemma}{lemma:rhomultiplicative} and \rref{Lemma}{lemma:clifford}.
    \end{enumerate}
\end{proof}

The following theorem tells us what a conjugation with a Clifford operator looks like for the corresponding additive symplectic self-orthogonal code. Here, a \emph{symplectic transformation} on \(\FF_2^{2n}\) is a change of basis that preserves the symplectic inner product. %This was proved more generally for stabiliser codes over \(\ZZ/d\ZZ\).
% (in our case, the standard symplectic inner product \(\langle v,w\rangle_s=v_1w_2+v_2w_2+\cdots+v_{2n-1}w_{2n}+v_{2n}w_{2n-1}\)).

\begin{theorem}[\cite{cliffordequiv}]\label{thm:cliffordequiv}
    Binary stabilisers are conjugated by a Clifford operator if and only if the corresponding symplectic self-orthogonal codes are the same up to a symplectic transformation.
\end{theorem}
\begin{proof}
    This was proved in \cite{cliffordequiv} more generally for stabiliser codes over \(\ZZ/d\ZZ\). If \(d=2\), these coincide with binary stabiliser codes, see also \rref{Remark}{remark:pauli}.
    
    An alternative proof can be found in \ref{app:skolem}.
\end{proof}

%This allows us to define equivalent additive codes if we want (?).

%\begin{definition}
%    Two symplectic self-orthogonal additive codes over \(\FF_{2^h}\) of length \(2n\) and size \(2^{h(n-k)}\) are equivalent if and only if the corresponding binary symplectic self-orthogonal additive codes of length \(2hn\) and size \(2^{h(n-k)}\), obtained by \(\phi_{n,h}\), can be converted into one another by a sequence of the following operations:
%    \begin{enumerate}[(i)]
%        \item A permutation on the \(n\) groups of \(2h\) positions.
%        \item A symplectic transformation on the length \(2h\) vectors within one of the \(n\) groups.
%    \end{enumerate}
%\end{definition}
%\begin{proof}
%    This follows from Theorem~\ref{thm:hequiv} and Theorem~\ref{thm:cliffordequiv}.
%\end{proof}

\begin{corollary}\label{cor:stabequiv}
    Two \([\![n,k,d]\!]_{2^h}\) codes are equivalent if and only if the stabiliser matrices of the corresponding \([\![hn,hk,d']\!]_2\) codes, obtained by \(\phi_{n,h}\), are the same up to:
    \begin{enumerate}[(i)]
        \item Row operations (i.e.\ left multiplication with an invertible \(h(n-k)\times h(n-k)\) matrix).
        \item Permutations on the \(n\) blocks of \(2h\) columns.
        \item Symplectic transformations on the \(2h\) columns of one of these \(n\) blocks.% (i.e.\ right multiplication with a \(2hn\times2hn\) matrix whose blocks are symplectic matrices of size \(2h\times2h\)).
    \end{enumerate}
    Here, a ``block'' denotes a set of \(h\) consecutive columns in the \(X\) part (obtained from one column by applying \(\phi_{n,h}\)) and the corresponding \(h\) columns in the \(Z\) part.
\end{corollary}
\begin{proof}
    This follows from \rref{Theorem}{thm:hequiv} and \rref{Theorem}{thm:cliffordequiv}. The generators are determined up to linear combinations, so the rows of the matrix are determined up to row operations.
\end{proof}

\begin{figure}[H]
    \centering
    \[
\begin{array}{c}
    \left(\begin{array}{@{}cccc@{}}
    &&&\\
    &\multicolumn{2}{c}{L}&\\
    &&&
    \end{array}\right)\\
    \\
    h(n-k)\times h(n-k)
\end{array}
\hspace{-10mm}
\begin{array}{c}
\cdot\\
\\
\\
\end{array}
\begin{array}{c}
    \circlearrowleft\qquad\qquad\qquad\circlearrowleft\\
    \left(\begin{array}{@{}ccc|ccc@{}}
    a_{11}&\cdots&a_{1,hn}&b_{11}&\cdots&b_{1,hn}\\
    \vdots&\ddots&\vdots&\vdots&\ddots&\vdots\\
    a_{r1}&\cdots&a_{r,hn}&b_{r1}&\cdots&b_{r,hn}
    \end{array}\right)\\
    \\
    h(n-k)\times2hn\\
    \\
\end{array}
\begin{array}{c}
\cdot\\
\\
\\
\end{array}
\begin{array}{c}
    \left(\begin{array}{@{}c@{}c@{}c|c@{}c@{}c@{}}
    R_1&&&S_1&&\\[-2mm]
    &\ddots&&&\ddots&\\
    &&R_n&&&S_n\\
    \hline
    T_1&&&U_1&&\\[-2mm]
    &\ddots&&&\ddots&\\[-2mm]
    &&T_n&&&U_n
    \end{array}\right)\\
    \\
    2hn\times2hn
\end{array}
\]
    \caption{Operations that give an equivalent stabiliser matrix, where \(L\) is invertible and each \(2h\times2h\) block matrix \(\begin{pmatrix}R_i&S_i\\T_i&U_i\end{pmatrix}\) is a symplectic matrix.}
\end{figure}

%\begin{example}
%    There are 13 nonequivalent \([\![8,0,d]\!]_2\) codes that can be mapped to a \([\![4,0,3]\!]_4\) code. I guess this one is unique? Investigate! See further...
%\end{example}

\section{The geometry of stabiliser codes over fields of even order}\label{sec:geometry}

Looking at stabiliser codes geometrically, provides a powerful tool to understand better what is going on. In fact, the previous results were found after investigating this link with finite geometry.

We refer to \cite{qeccandtheirgeom} for an overview of some links between quantum error-correcting codes and finite geometry.

For convenience, we will treat stabiliser codes with the same stabiliser matrix as the same code (though they differ in phase factors).

\subsection{Binary stabiliser codes as quantum sets of lines}

Let \(\Q\) be a binary stabiliser code with minimum distance \(d\geq2\). Recall that \(\Q\) can be represented by a stabiliser matrix, whose rows are the vectors \(\tau(S)\) for every \(S\) in a given generating set of the stabiliser (see \rref{Definition}{def:tau}). We can interpret the columns of this matrix as points of the projective space \(\PG(r-1,q)\). For each \(i=1,\dots,n\), the span of the \(i\)-th column and the \((n+i)\)-th column is a line (and not a point, because the columns are different as \(\N(\S)\) has no elements of weight one, see \cite[Lemma~3.6]{qeccandtheirgeom}). This set of lines defines a so-called \emph{quantum set of lines}, as defined below. This concept was introduced in \cite{GGMG} and refined in \cite{qeccandtheirgeom}. Vice versa, a quantum set of lines defines a stabiliser code by considering them as columns of a stabiliser matrix. This correspondence is made concrete in \rref{Theorem}{thm:quantumlines} below.

\begin{definition}\label{def:qsol}
    A \emph{quantum set of \(n\) lines} in \(\PG(r-1,2)\) is a set \(X\) of \(n\) lines spanning \(\PG(r-1,2)\) such that every codimension two subspace is skew to an even number of the lines.

    
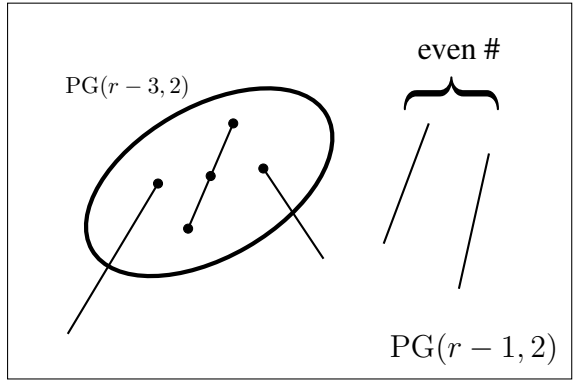
\begin{figure}[H]
    \centering
    \begin{tikzpicture}[scale=2]
\draw[rotate=120,ultra thick] (1.8,.3) ellipse (.5 and .9) {};
\draw[thick] (-2.1,.4) -- (-1.5,1.4);
\draw[thick] (-1.3,1.1) -- (-1,1.8);
\draw[thick] (-.8,1.5) -- (-.4,.9);
\draw[thick] (0,1) -- (.3,1.8);
\draw[thick] (.5,.7) -- (.7,1.6);
\draw[fill] (-1.5,1.4) {} circle (.03)
(-1.3,1.1) {} circle (.03)
(-1.15,1.45) {} circle (.03)
(-1,1.8) {} circle (.03)
(-.8,1.5) {} circle (.03);
\node[scale=.8] at (-1.7,2.04) {\small\(\PG(r-3,2)\)};
\node at (.6,.3) {\(\PG(r-1,2)\)};
\node[scale=3,rotate=-90] at (.45,2) {\(\{\)};
\node at (.5,2.3) {even \#};
\draw (-2.5,.1) rectangle (1.25,2.6);
\end{tikzpicture}
    %\caption{A quantum set of lines has the property that every codimension two space is skew to an even number of its lines.}
    \end{figure}

    If \(k=0\), we define the \emph{minimum distance} \(d\) of \(X\) as the minimum size of a set of dependent points on distinct lines of \(X\).
    If \(k>0\), we define \(d\) as the minimum size of a set of dependent points on distinct lines of \(X\) such that there is no hyperplane containing these points and all the lines of \(X\) that do not contain one of these points.
\end{definition}

\begin{theorem}[{\cite[Proposition~1.13]{GGMG}}]\label{thm:quantumlines}
    Let \(d\geq2\). There is a bijection between:% bijective correspondence between:% the equivalence classes of:

    %The following are equivalent:
\begin{enumerate}[(i)]
    \item \([\![n,k,d]\!]_2\) stabiliser codes.
    %\item A symplectic self-orthogonal \([2n,n-k]_2\)-code \(C\) with \(\mathrm{swt}(C^{\perp_s})=d\) if \(k=0\) and \(\mathrm{swt}(C^{\perp_s}\setminus C)=d\) if \(k>0\).
    \item Quantum sets of \(n\) lines in \(\PG(n-k-1,2)\) with minimum distance \(d\).
\end{enumerate}
    Moreover, binary stabiliser codes are equivalent if and only if their quantum sets of lines are projectively equivalent.
\end{theorem}

The last sentence of the above theorem ensures that these lines really are the same objects as binary stabiliser codes.
%\begin{proof}
%    We first show that neither of the three operations of Definition~\ref{def:equivalent} changes the geometry. The first one corresponds to a permutation of the operators \(X,Z\) and \(XZ\).
%    We now show that two codes with the same geometry are equivalent. Any collineation of \(\PG(r-1,q)\) induced by an element of \(\mathrm{GL}(\mathbb{F}_q^r)\), corresponds to a change of generators of the stabiliser and hence leaves the stabiliser invariant. TODO
%\end{proof}

\begin{example}
    Consider the matrix
    \[\left(\begin{array}{@{}c|c@{}}I_8&A\end{array}\right)=\left(\begin{array}{@{}cccccccc|cccccccc@{}}
    1&0&0&0&0&0&0&0& 0&1&0&0&0&0&0&1\\
    0&1&0&0&0&0&0&0& 1&0&1&0&0&0&0&0\\
    0&0&1&0&0&0&0&0& 0&1&0&1&0&0&0&0\\
    0&0&0&1&0&0&0&0& 0&0&1&0&1&0&0&0\\
    0&0&0&0&1&0&0&0& 0&0&0&1&0&1&0&0\\
    0&0&0&0&0&1&0&0& 0&0&0&0&1&0&1&0\\
    0&0&0&0&0&0&1&0& 0&0&0&0&0&1&0&1\\
    0&0&0&0&0&0&0&1& 1&0&0&0&0&0&1&0
    \end{array}\right)\]
    %\robin{will be replaced by\[\left(\begin{array}{@{}cc|cc|cc|cc|cc|cc|cc|cc@{}}    1&0&0&1&0&0&0&0&0&0&0&0&0&0&0&1\\    0&1&1&0&0&1&0&0&0&0&0&0&0&0&0&0\\    0&0&0&1&1&0&0&1&0&0&0&0&0&0&0&0\\    0&0&0&0&0&1&1&0&0&1&0&0&0&0&0&0\\    0&0&0&0&0&0&0&1&1&0&0&1&0&0&0&0\\    0&0&0&0&0&0&0&0&0&1&1&0&0&1&0&0\\    0&0&0&0&0&0&0&0&0&0&0&1&1&0&0&1\\    0&1&0&0&0&0&0&0&0&0&0&0&0&1&1&0    \end{array}\right)\]}
    where \(A\) is the adjacency matrix of the \(8\)-cycle. It is the stabiliser matrix of an \([\![8,0,3]\!]_2\) code (more precisely, it is the \emph{graph state} of the \(8\)-cycle). The corresponding quantum set of lines looks like this:
    
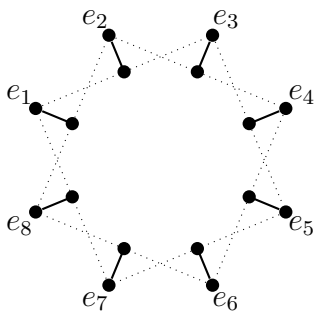
\begin{figure}[H]
        \centering
        \begin{tikzpicture}[scale=.9, remember picture]
    \path[every node/.append style={circle, fill=black, minimum size=5pt, label distance=-1pt, inner sep=0pt}]
    (157.5:2) node[label={157.5:\(e_1\)}] (1) {}
    (112.5:2) node[label={112.5:\(e_2\)}] (2) {}
    (67.5:2) node[label={67.5:\(e_3\)}] (3) {}
    (22.5:2) node[label={22.5:\(e_4\)}] (4) {}
    (337.5:2) node[label={337.5:\(e_5\)}] (5) {}
    (292.5:2) node[label={292.5:\(e_6\)}] (6) {}
    (247.5:2) node[label={247.5:\(e_7\)}] (7) {}
    (202.5:2) node[label={202.5:\(e_8\)}] (8) {};
    \draw[dotted] (1) -- coordinate (13) (3) -- coordinate (35) (5) -- coordinate (57) (7) -- coordinate (17) (1);
    \draw[dotted] (2) -- coordinate (24) (4) -- coordinate (46) (6) -- coordinate (68) (8) -- coordinate (28) (2);
    \path[every node/.append style={circle, fill=black, minimum size=5pt, label distance=-1pt, inner sep=0pt}]
    (28) node {} (13) node {} (24) node {} (35) node {} (46) node {} (57) node {} (68) node {} (17) node {};
    \draw[thick] (1) -- (28) (2) -- (13) (3) -- (24) (4) -- (35) (5) -- (46) (6) -- (57) (7) -- (68) (8) -- (17);
\end{tikzpicture}
    \end{figure}
    Its minimum distance is \(d=3\), since every two lines are in general position (\(d>2\)) and the three points \(e_1\), \(e_3\) and \(e_1+e_3\) lie on different lines and are linearly dependent (\(d\leq3\)).
\end{example}

\subsection{Stabiliser codes over fields of even order as quantum sets of sets of lines}

We introduce the following new concept, extending \rref{Definition}{def:qsol} to arbitrary \(q=2^h\).

\begin{definition}
    A \emph{quantum set of \(n\) sets of \(h\) lines} in \(\PG(r-1,2)\) is a partitioning \(\X\) of a quantum set of \(hn\) lines in \(\PG(r-1,2)\) into \(n\) subsets of \(h\) lines that span a projective \((2h-1)\)-space. Denote these projective \((2h-1)\)-spaces by \(\pi_1,\pi_2,\dots,\pi_n\).

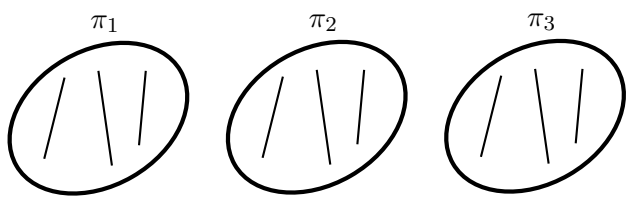
    \begin{figure}[H]
    \centering
    \begin{tikzpicture}[scale=1.8]
\draw[rotate=120,ultra thick] (0,0)+(.8,1.4) ellipse (.5 and .7) node (A) {};
\draw[rotate=120,ultra thick] (0,0) ellipse (.5 and .7) {};
\draw[rotate=120,ultra thick] (0,0)+(-.8,-1.4) ellipse (.5 and .7) node (B) {};
\draw[thick] (-.4,-.3) -- (-.25,.3) (.1,-.35) -- (0,.35) (.3,-.2) -- (.35,.35);
\draw[thick] (A)+(-.4,-.3) -- +(-.25,.3) (A)+(.1,-.35) -- +(0,.35) (A)+(.3,-.2) -- +(.35,.35);
\draw[thick] (B)+(-.4,-.3) -- +(-.25,.3) (B)+(.1,-.35) -- +(0,.35) (B)+(.3,-.2) -- +(.35,.35);
\draw (A)+(.05,.7) node {\small\(\pi_1\)};
\draw (.05,.7) node {\small\(\pi_2\)};
\draw (B)+(.05,.7) node {\small\(\pi_3\)};
\end{tikzpicture}
    \end{figure}
    
    If \(k=0\), we define the \emph{minimum distance} \(d\) of \(\X\) as the minimum size of a set of dependent points in distinct \(\pi_i\)'s.
    If \(k>0\), we define \(d\) as the minimum size of a set of dependent points in distinct \(\pi_i\)'s such that there is no hyperplane containing these points and all \(\pi_i\)'s that do not contain one of these points.
\end{definition}

Recall that the stabiliser matrix of an \([\![n,k,d]\!]_{2^h}\) stabiliser code is an \(h(n-k)\times2n\) matrix \(G\) with entries in \(\mathbb{F}_{2^h}\). If we replace every column of this matrix by \(h\) columns according to a trace-orthogonal basis of \(\mathbb{F}_{2^h}\cong\mathbb{F}_2^h\) over \(\mathbb{F}_2\), we obtain an \(h(n-k)\times2hn\) matrix \(G'\) with entries in \(\mathbb{F}_2\). Let \(\pi_i\) be the subspace of \(\PG(h(n-k)-1,2)\) that is spanned by the \(h\) columns from the original \(i\)-th column and the \(h\) columns from the original \((n+i)\)-th column.

\begin{lemma}[{\cite[Lemma~5.8]{qeccandtheirgeom}}]\label{lemma:generalposition}
    The subspace \(\pi_i\) is a \((2h-1)\)-dimensional subspace for all \(i=1,\dots,n\) if and only if the minimum nonzero weight of \(\N(\S)\) is at least two.
\end{lemma}
%Suppose, by contradiction, that some \(h\)-subset of lines spans a space that does not have full dimension \(2h-1\). This means that the \(2h\) columns are linearly dependent. So we can construct a column whose support only lies in these \(2h\) positions and that is symplectic orthogonal to every row of the matrix. If we replace the groups of \(h\) columns by one column in \(\mathbb{F}_{2^h}\) again, we get a codeword with only nonzero elements in a column \(i\) and the corresponding column \(n+i\). In other words, we have an element of weight \(1\) in \(\mathrm{Centraliser}(\S)\). This is a contradiction.

We can now state \rref{Theorem}{thm:main} geometrically, thereby generalising \rref{Theorem}{thm:quantumlines}.

\begin{theorem}\label{thm:maingeom}
    Let \(d\geq2\). There is a bijection between:% following are equivalent:
    \begin{enumerate}[(i)]
        \item \([\![n,k,d]\!]_{2^h}\) stabiliser codes.
        \item Quantum sets of \(n\) sets of \(h\) lines \(\X\) in PG\((h(n-k)-1,2)\) with minimum distance \(d\).
    \end{enumerate}
\end{theorem}
\begin{proof}
    Let \(G\) be the stabiliser matrix of an \([\![n,k,d]\!]_{2^h}\) code. If we apply \(\phi_{n,h}\) on each of its rows, we obtain the stabiliser matrix \(G'\) of an \([\![hn,hk,d']\!]_2\) code by \rref{Theorem}{thm:main}.
    By \rref{Theorem}{thm:quantumlines}, the rows of \(G'\) correspond to a quantum set of lines \(X\) in \(\PG(h(n-k)-1)\) with \(|X|=hn\). The lines of \(X\) are spanned by pairs of points whose coordinates are the \(i\)-th and \((hn+i)\)-th column of \(G'\).
    If we partition the quantum set of lines \(X\) into subsets \(X_i\) of the \(h\) lines that result from the \(i\)-th and \((n+i)\)-th column in \(G\), we get a quantum set of \(n\) sets of \(h\) lines. Indeed, by \rref{Lemma}{lemma:generalposition}, and since \(d\geq2\), the lines of every such \(h\)-set span a \((2h-1)\)-space.
    
    Conversely, every quantum set of \(n\) sets of \(h\) lines \(\X\) in PG\((h(n-k)-1,2)\) with \(|\X|=n\), can be mapped to the stabiliser matrix of an \([\![hn,hk,d']\!]_2\) code, after which one can apply \(\phi_{n,h}^{-1}\) (i.e.\ group every \(h\) columns together) to obtain a stabiliser matrix of an \([\![n,k,d]\!]_{2^h}\) code.

    We now prove that the minimum distances are the same. (For this part, we follow the proof of \cite[Theorem~3.8]{qeccandtheirgeom} with some adjustments.) Let \(C\) be the code with generator matrix \(G\). By \rref{Theorem}{thm:ketkar}, \(\mathrm{swt}(C^{\perp_s})=d\) if \(k=0\) and \(\mathrm{swt}(C^{\perp_s}\setminus C)=d\) if \(k>0\).

    If \(k=0\), then \(d=\mathrm{swt}(C^{\perp_s})\). Choose \((a|b)\in C^{\perp_s}\) with \(\mathrm{swt}(a|b)=d\) and write its entries as \(a_i=\sum_{j=1}^ha_{ij}e_{j}\) and \(b_i=\sum_{j=1}^hb_{ij}e_{j}\). Let \(W\) be the set of indices that contribute to the symplectic weight, i.e.\
    \[W=\left\{i\in\{1,\dots,n\}\mid(a_i,b_i)\neq(0,0)\right\}.\]
    Let \(v_i\) denote the \(i\)-th column of \(G\), and write it as a span of the corresponding \(h\) columns in \(G'\), \(v_i=\sum_{j=1}^hv_{ij}e_{j}\).
    Since \((a|b)\) is in \(C^{\perp_s}\), we have
    \begin{align*}
        \Vec{0}&=\sum_{i\in W}\tr(a_iv_{n+i}+v_ib_i)\\
        &=\sum_{i\in W}\sum_{j,l=1}^h(a_{ij}v_{n+i,l}+v_{il}b_{ij})\mathrm{tr}(e_je_l)\\
        &=\sum_{i\in W}\sum_{j=1}^h(a_{ij}v_{n+i,j}+v_{ij}b_{ij}).
    \end{align*}
    Each summand \(\sum_{j=1}^h(a_{ij}v_{n+i,j}+v_{ij}b_{ij})\) is a point in the space \(\pi_i\) spanned by the \(i\)-th set of \(h\) lines of \(\X\). Thus, there are \(|W|=\mathrm{swt}(a|b)=d\) points on distinct \((2h-1)\)-spaces spanned by an \(h\)-set that are dependent. Conversely, every point in \(\pi_i\) can be written as \(\sum_{j=1}^h(a_{ij}v_{n+i,j}+v_{ij}b_{ij})\) for some \(a_i\) and \(b_i\). Hence, the existence of a set of \(d\) dependent points in distinct \(\pi_i\)'s also implies the existence of a vector \((a|b)\in C^{\perp_s}\) with \(\mathrm{swt}(a|b)=d\).
    
    If \(k>0\), then \(d=\mathrm{swt}(C^{\perp_s}\setminus C)\). Choose \((a|b)\in C^{\perp_s}\setminus C\) with \(\mathrm{swt}(a|b)=d\). Since \((a|b)\in C^{\perp_s}\), the above reasoning implies that there are \(d\) points on distinct \((2h-1)\)-spaces spanned by an \(h\)-set that are dependent. However, since \((a|b)\notin C\), we should disregard the case where \((a|b)=x\cdot G\) for some \(x\in\mathbb{F}_2^{h(n-k)}\). Consider the hyperplane \(\pi:x\cdot X=0\) in \(\PG(h(n-k)-1,2)\). If \(i\in W\), then \(x\cdot\sum_{j=1}^h(a_{ij}v_{n+i,j}+v_{ij}b_{ij})=\sum_{j=1}^h(a_{ij}b_{ij}+a_{ij}b_{ij})=0\), so the dependent points are contained in \(\pi\). If \(i\notin W\), then \(a_i=x\cdot v_i=0\) and \(b_i=x\cdot v_{n+i}=0\) are both zero, which is the case if and only if the \(i\)-th \((2h-1)\)-space is contained in \(\pi\).
\end{proof}

\begin{example}
    Here is an example of a pairing of lines to make a \([\![4,0,3]\!]_4\) code.
    
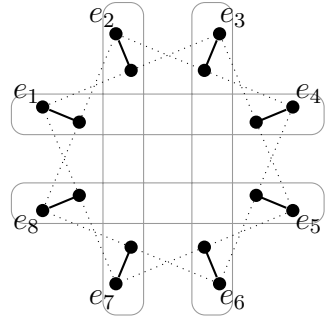
\begin{figure}[H]
        \centering
        \begin{tikzpicture}[scale=.9, remember picture]
    \draw[black!40,rounded corners=5pt]
    (157.5:2.5)+(0,-.6) rectangle (22.5:2.5)
    (202.5:2.5)+(0,.6) rectangle (337.5:2.5)
    (112.5:2.5)+(.6,0) rectangle (247.5:2.5)
    (67.5:2.5)+(-.6,0) rectangle (292.5:2.5);
    \path[every node/.append style={circle, fill=black, minimum size=5pt, label distance=-1pt, inner sep=0pt}]
    (157.5:2) node[label={157.5:\(e_1\)}] (1) {}
    (112.5:2) node[label={112.5:\(e_2\)}] (2) {}
    (67.5:2) node[label={67.5:\(e_3\)}] (3) {}
    (22.5:2) node[label={22.5:\(e_4\)}] (4) {}
    (337.5:2) node[label={337.5:\(e_5\)}] (5) {}
    (292.5:2) node[label={292.5:\(e_6\)}] (6) {}
    (247.5:2) node[label={247.5:\(e_7\)}] (7) {}
    (202.5:2) node[label={202.5:\(e_8\)}] (8) {};
    \draw[dotted] (1) -- coordinate (13) (3) -- coordinate (35) (5) -- coordinate (57) (7) -- coordinate (17) (1);
    \draw[dotted] (2) -- coordinate (24) (4) -- coordinate (46) (6) -- coordinate (68) (8) -- coordinate (28) (2);
    \path[every node/.append style={circle, fill=black, minimum size=5pt, label distance=-1pt, inner sep=0pt}]
    (28) node {} (13) node {} (24) node {} (35) node {} (46) node {} (57) node {} (68) node {} (17) node {};
    \draw[thick] (1) -- (28) (2) -- (13) (3) -- (24) (4) -- (35) (5) -- (46) (6) -- (57) (7) -- (68) (8) -- (17);
\end{tikzpicture}
    \end{figure}
    The minimum distance is \(d=3\), since every two solids (projective \(3\)-spaces) span the entire space (\(d>2\)) and the three points \(e_1\), \(e_3\) and \(e_1+e_3\) lie in different solids and are linearly dependent (\(d\leq3\)). We will prove in \rref{Section}{sec:applications} that this code is unique.
\end{example}

This provides a geometrical interpretation of the map \(\Psi_{n,h}\) defined earlier. It sends a quantum set of sets of lines to its underlying quantum set of lines. The inverse map \(\Psi_{n,h}^{-1}\) groups a quantum set of lines into sets of \(h\) lines, but only if the lines in one group are in general position.

%\begin{corollary}
%    There exists an \([\![n,k,d]\!]_{2^h}\) code for some \(d\) if and only if there exists an \([\![hn,hk,d']\!]_2\) code for some \(d'\), for which we can regroup the corresponding quantum set of lines into \(h\)-sets of lines that span a \((2h-1)\)-space. Moreover, \(d'/h\leq d\leq d'\).
%\end{corollary}
%\begin{proof}
%    The first statement follows from the preceding theorem.

%    \(d'\leq hd\). There are \(d\) dependent points on distinct \((2h-1)\)-spaces spanned by the \(h\)-sets of the \([\![n,k,d]\!]_{2^h}\) code. Each of those points can be written as the span of at most \(h\) points on distinct lines of that \(h\)-set. The union of these points if therefore a set of at most \(hd\) dependent points on distinct lines of the quantum set of lines from the \([\![hn,hk,d']\!]_2\) code.

%    \(d\leq d'\). There are \(d'\) dependent points on distinct lines of the quantum set of lines from the \([\![hn,hk,d']\!]_2\) code. In particular, each of these points also lies in a \((2h-1)\)-space spanned by an \(h\)-set. They cannot all lie in the same \((2h-1)\)-space, since the lines in there are in general position. So there is a set of less than \(d'\) dependent points in distinct \((2h-1)\)-spaces spanned by the \(h\)-sets of the \([\![n,k,d]\!]_{2^h}\) code.
    %Alternatively, we could look at the corresponding symplectic self-orthogonal code.
%\end{proof}

\subsection{Equivalent quantum sets of sets of lines}

Although stabiliser codes over finite fields of characteristic two can be seen as quantum sets of sets of lines, the equivalence of codes is not necessarily translated into projective equivalence. We can do more. Conjugation with a Clifford operator now corresponds to a symplectic transformation of each of the \((2h-1)\)-spaces. To understand this better, we introduce the notion of a symplectic polar space.

\begin{definition}
    Let \(V\) be a vector space. A \emph{bilinear form} on \(V\) is a map \(f:V^2\to V\) that is linear in both arguments. A \emph{totally isotropic subspace} of \(f\) is a subspace \(W\leqslant V\) such that \(f(v,w)=0\) for any \(v,w\in W\).
    
    An \emph{symplectic form} (or \emph{alternating form}) on \(V\) is a bilinear form \(f:V^2\to V\) such that \(f(v,v)=0\) for all \(v\in V\). It is \emph{degenerate} if there exists a \(v\in V\setminus\{0\}\) such that \(f(v,\cdot)\equiv0\). A \emph{symplectic transformation} is a change of basis of \(V\) that preserves \(f\).
\end{definition}

If \(V=\FF_2^n\), one can show that \(f\) is symplectic if and only if \(f(v,w)=v^TAw\) for some symmetric \((n\times n)\)-matrix \(A\) with zeroes on the diagonal. The matrix \(A\) is singular if and only if \(f\) is degenerate. %If \(f\) is nondegenerate, then with respect to a well-chosen basis,
%\[A=\begin{pmatrix}0&1&&&&&\\1&0&&&&&\\&&0&1&&&\\&&1&0&&&\\&&&&\ddots&&\\&&&&&0&1\\&&&&&1&0\end{pmatrix}.\]
%In this case, \(f\) corresponds to the standard symplectic inner product defined earlier.

\begin{definition}
    A \emph{symplectic polar space of rank \(h\)}, denoted by \(W(2h-1,2)\), is the set of totally isotropic subspaces of a nondegenerate symplectic form \(f\) on \(\FF_2^{2h}\).
    We associate with it a \emph{symplectic polarity} on \(\PG(2h-1,2)\), defined by
    \[\pi\mapsto\pi^\perp:=\{v\in\FF_2^{2h}\mid f(v,w)=0\text{ for all }w\in\pi\}.\]
    %It is called \emph{degenerate} if \(f\) is degenerate.

    %A \emph{totally isotropic subspace} is a subspace \(\pi\) of \(\PG(2h-1,2)\) such that \(\pi\subseteq\pi^\perp\), i.e.\ such that \(f(v,w)=0\) for every \(v,w\in\pi\).
\end{definition}

%\begin{lemma}
%    A line \(\ell\) is totally isotropic if and only if \(\ell\cap\ell^\perp\neq\{0\}\).
%\end{lemma}
%\begin{proof}
%    It suffices to prove that \(\ell\cap\ell^\perp\) is never a point. Suppose it were. Choose any \(p\in\ell\setminus\ell^\perp\). Then \(\ell^\perp\subseteq p^\perp\), but also \(p\in p^\perp\), so \(\ell\in p^\perp\) and thus \(p\in\ell^\perp\), a contradiction.
%\end{proof}

%\begin{lemma}
%    Let \(\ell_1,\ell_2\) be two nonintersecting lines in \(\PG(3,2)\). There is a unique symplectic polarity \(\theta\) on \(\PG(3,2)\) such that \(\ell_1^\theta=\ell_2\).
%\end{lemma}
%\begin{proof}
%    Choose coordinates \(e_1,e_2\in\ell_1\) and \(e_3,e_4\in\ell_2\). The standard nondegenerate alternating form \[f(v,w)=v_1w_2+v_2w_1+v_3w_4+v_4w_3\] is the unique nondegenerate alternating form such that \(\ell_1^\theta=\ell_2\) for the associated symplectic polarity \(\theta\).
%\end{proof}

\begin{lemma}\label{lemma:lines}
    Let \(\ell_1,\dots,\ell_h\) be \(h\) lines that span a \((2h-1)\)-space. There is a unique symplectic polarity on this space such that \(\ell_i^\perp=\langle\ell_j\mid i\neq j\rangle\) for all \(i=1,\dots,h\).
\end{lemma}
\begin{proof}
    Choose a basis \(\{e_1,\dots,e_{2h}\}\) of \(\FF_2^{2h}\) such that \(\ell_i=\langle e_i,e_{h+i}\rangle\) for all \(i\). The symplectic inner product \[f(v,w)=\langle v,w\rangle_s=v_1w_{h+1}+v_{h+1}w_1+\cdots+v_hw_{2h}+v_{2h}w_h\] is the unique nondegenerate symplectic form for which the associated symplectic polarity fulfils the condition in the statement.
\end{proof}

With regard to the previous lemma, we use the notation \(f_{\{\ell_1,\dots,\ell_h\}}\) for the unique symplectic form determined by the lines \(\ell_1,\dots,\ell_h\).

%\begin{lemma}
%    Let \(\ell_1,\ell_2\) be two nonintersecting lines in \(\PG(3,2)\). There is a unique symplectic polarity \(\theta\) such that \(\ell_1^\theta=\ell_2\).
%    The totally isotropic lines of \(\theta\) are precisely those lines that either intersect both \(\ell_1\) and \(\ell_2\) or are skew to both \(\ell_1\) and \(\ell_2\).
%    The hyperbolic lines are precisely those lines that intersect either \(\ell_1\) or \(\ell_2\) but not both.
%\end{lemma}
%\begin{proof}
%    Let \(\ell\) be any line. First suppose that it intersects one of the lines in at least a point \(p\). Note that \(p^\perp\) is the plane spanned by \(p\) and the other line. If \(\ell\) lies in that plane, it is a totally isotropic line, and if it does not, it is hyperbolic. Now suppose that \(\ell\) is disjoint with the two lines. Let \(p\in\ell\) and consider \(p^\perp\). There is a unique line through \(p\) that intersects both \(\ell_1\) and \(\ell_2\). It is contained in a unique plane that does not contain these two lines fully. This plane is therefore equal to \(p^\perp\) and must contain \(\ell\). So \(\ell\) is a totally isotropic line. 
%\end{proof}

%\begin{example}
%    Let \(\theta\) be a symplectic polarity of \(\PG(3,2)\). It has \(15\) totally isotropic lines and \(20\) hyperbolic lines.
%\end{example}

%\subsubsection{Equivalent quantum sets of sets of lines}

%Let \(\X\) be a quantum set of \(h\)-sets of lines and let \(\{\ell_1,\dots,\ell_h\}\) be one such \(h\)-set of lines from \(\X\). Let \(\pi\) be the projective \((2h-1)\)-space spanned by them.

\begin{theorem}\label{thm:qsohlequiv}
    Stabiliser codes over fields of even order are equivalent if and only if their quantum sets of sets of lines are the same up to a projective equivalence and, for each set of lines \(\{\ell_1,\dots,\ell_h\}\), a symplectic transformation on \(\langle\ell_1,\dots,\ell_h\rangle\) with respect to \(f_{\{\ell_1,\dots,\ell_h\}}\).
\end{theorem}
\begin{proof}
    Recall that every two consecutive columns in the stabiliser matrix correspond to a line, and every block of \(2h\) columns determines a set of \(h\) lines.
    The statement then follows from \rref{Corollary}{cor:stabequiv}:
    \begin{enumerate}[(i)]
        \item Row operations on the stabiliser matrix are equivalent to projective transformations.
        \item Permutations of the blocks are exactly the permutations on the \((2h-1)\)-spaces.
        \item A symplectic transformation on the \(2h\) columns of one block is a change of basis that preserves the symplectic inner product. But the lines are ordered in such a way that \(\ell_i=\langle e_i,e_{h+i}\rangle\), where \(\{e_1,\dots,e_h\}\) is the standard basis of \(\langle\ell_1,\dots,\ell_h\rangle\). Following the proof of \rref{Lemma}{lemma:lines}, \(f_{\{\ell_1,\dots,\ell_h\}}\) is therefore also the symplectic inner product. So a symplectic transformation in the sense of \rref{Corollary}{cor:stabequiv} is the same as a symplectic transformation with respect to \(f_{\{\ell_1,\dots,\ell_h\}}\).
    \end{enumerate}
\end{proof}

\subsection{Stabiliser codes over fields of even order as quantum sets of symplectic polar spaces}

The following lemma is a special case of Witt's theorem. A \emph{hyperbolic} line is a line that is not totally isotropic.

\begin{lemma}[{\cite[Theorem~3.4]{witt}}]\label{lemma:witt}
    The group of symplectic transformations %on \(W(2h-1,2)\) 
    acts transitively on the totally isotropic (resp.\ hyperbolic) lines.
\end{lemma}

Note that a recursive application of this lemma allows one to prove that the group of symplectic transformations on \(W(2h-1,2)\) acts transitively on the sets of \(h\) lines \(\{\ell_1,\dots,\ell_h\}\) for which \(\ell_i^\perp=\langle\ell_j\mid i\neq j\rangle\).

%Therefore, we can always replace a quantum pair by any other pair of conjugated hyperbolic lines.

%\begin{corollary}
%    The group of symplectic transformations on \(W(2h-1,2)\) acts transitively on the sets of \(h\) lines \(\{\ell_1,\dots,\ell_h\}\) for which \(\ell_i^\perp=\langle\ell_j\mid i\neq j\rangle\).
%\end{corollary}
%\begin{proof}
%    Let \(\ell_1,\dots,\ell_h\) be \(h\) lines such that \(\ell_i^\perp=\langle\ell_j\mid i\neq j\rangle\) and let \(m_1,\dots,m_h\) be another \(h\) lines such that \(m_i^\perp=\langle m_j\mid i\neq j\rangle\). We show that there exists a symplectic transformation that maps each \(\ell_i\) onto \(m_i\), by induction on \(h\). For \(h=1\), there is only one line and the identity map suffices. Suppose that the statement holds for \(h\) lines. We want to prove it for \(h+1\) lines. By Witt's theorem, there exists a symplectic transformation that maps \(\ell_1\) onto \(m_1\). Now apply the induction hypothesis on the \((2h-1)\)-space \(m_1^\perp\) and combine the two symplectic transformations.
%\end{proof}

\begin{lemma}\label{lemma:hyperbolic}
    A line \(\ell\) is hyperbolic if and only if \(\ell\cap\ell^\perp=\{0\}\).
\end{lemma}
\begin{proof}
    It suffices to prove that \(\ell\cap\ell^\perp\) is never a point. Suppose it were. Choose any \(p\in\ell\setminus\ell^\perp\). Then \(\ell^\perp\subseteq p^\perp\), but also \(p\in p^\perp\), so \(\ell\subseteq p^\perp\) and thus \(p\in\ell^\perp\), a contradiction.
\end{proof}

\begin{theorem}\label{thm:evenproperty}
    Let \(\ell_1,\dots,\ell_h\) be \(h\) lines such that \(\ell_i^\perp=\langle\ell_j\mid i\neq j\rangle\) and let \(\pi\) be a codimension two space. The property ``\(\pi\) is skew to an even number of the lines \(\ell_1,\dots,\ell_h\)'' is preserved under symplectic transformations with respect to \(f_{\{\ell_1,\dots,\ell_h\}}\).
\end{theorem}
\begin{proof}
    From the proof of \rref{Theorem}{thm:quantumlines} (see \cite[Proposition~1.13]{GGMG}), this property is equivalent to saying that a certain pair of rows of the associated stabiliser, is symplectic orthogonal. By \rref{Corollary}{cor:stabequiv}%, with just one block (\(n=1\))
    , this property is preserved under symplectic transformations. Part (iii) of the proof of \rref{Theorem}{thm:qsohlequiv} tells us that these are exactly the symplectic transformations with respect to \(f_{\{\ell_1,\dots,\ell_h\}}\).
\end{proof}

\begin{theorem}\label{thm:qsosps}
    Let \(\pi\) be a subspace of \(\PG(2h-1,2)\) of codimension 0,1 or 2. Let \(\ell_1,\dots,\ell_h\) be \(h\) lines such that \(\ell_i^\perp=\langle\ell_j\mid i\neq j\rangle\). The following are equivalent.
    \begin{enumerate}[(i)]
        \item \(\pi\) is skew to an even number of the lines \(\ell_1,\dots,\ell_h\).
        %\item \(\pi\) contains an odd number of maximal totally isotropic subspaces of \(W(2h-1,2)\).
        %\item \(\pi\) contains a maximal totally isotropic subspace of \(W(2h-1,2)\).
        \item \(\pi^\perp\) is totally isotropic.
    \end{enumerate}
\end{theorem}
\begin{proof}
    If the codimension of \(\pi\) is 0 or 1, both properties are always true, since \(\pi\) is skew to no line, and every point is totally isotropic. So assume that \(\pi^\perp\) is a line.
    %If the subspace is the whole space, it is skew to \(0\) lines and contains \((2+1)(2^2+1)\cdots(2^h+1)\equiv1\pmod{2}\) generators. If it is a hyperplane, it is also skew to \(0\) lines, and it contains \((2+1)(2^2+1)\cdots(2^{h-1}+1)\equiv1\pmod{2}\) generators, all through the same point.

    If \(\pi^\perp\) is a totally isotropic line, choose two points \(p,p'\in\pi^\perp\subseteq\pi\) and a point \(q\notin p^\perp\). Then \(pq\) is a hyperbolic line and there exists a symplectic transformation \(\sigma\) such that \(\sigma(\ell_1)=pq\) by \rref{Lemma}{lemma:witt}. Moreover, \((pq)^\perp=\sigma(\ell_1)^\perp=\langle\sigma(\ell_2),\dots,\sigma(\ell_h)\rangle\). Now, \((pq)^\perp\cap\pi=(pq)^\perp\cap p'^\perp\) is a hyperplane of \((pq)^\perp\), so all lines \(\sigma(\ell_1),\dots,\sigma(\ell_h)\) intersect \(\pi\). Thus, \(\pi\) is skew to none of the lines \(\sigma(\ell_1),\dots,\sigma(\ell_h)\). By \rref{Theorem}{thm:evenproperty}, \(\pi\) is skew to an even number of the lines \(\ell_1,\dots,\ell_h\).

    If \(\pi^\perp\) is a hyperbolic line, then \(\pi\cap\pi^\perp=\{0\}\) by \rref{Lemma}{lemma:hyperbolic}. There exists a symplectic transformation \(\sigma\) such that \(\sigma(\ell_1)=\pi^\perp\) by \rref{Lemma}{lemma:witt}. Then \(\pi=\langle \sigma(\ell_2),\dots,\sigma(\ell_h)\rangle\) and hence, \(\pi\) is skew to exactly one of the lines \(\sigma(\ell_1),\dots,\sigma(\ell_h)\), namely \(\sigma(\ell_1)\). By \rref{Theorem}{thm:evenproperty}, \(\pi\) is skew to an odd number of the lines \(\ell_1,\dots,\ell_h\).

\end{proof}

\begin{definition}\label{def:qsosps}
    A \emph{quantum set of \(n\) symplectic polar spaces of rank \(h\)} in \(\PG(r-1,2)\) is a set \(X\) of \(n\) projective \((2h-1)\)-spaces spanning \(\PG(r-1,2)\), each equipped with a symplectic polarity with the following property: every codimension two subspace intersects an even number of the elements of \(X\) in a subspace \(\pi\) %that satisfies one of the equivalent properties of Theorem~\ref{thm:qsosps}. 
    for which \(\pi^\perp\) is totally isotropic.

    If \(k=0\), we define the \emph{minimum distance} \(d\) of \(X\) as the minimum size of a set of dependent points in distinct \((2h-1)\)-spaces of \(X\).
    If \(k>0\), we define \(d\) as the minimum size of a set of dependent points on distinct \((2h-1)\)-spaces of \(X\) such that there is no hyperplane containing these points and all \((2h-1)\)-spaces of \(X\) that do not contain one of these points.
\end{definition}

\begin{theorem}\label{thm:final}
    Let \(d\geq2\). There is a bijection between:
    \begin{enumerate}[(i)]
        \item \([\![n,k,d]\!]_{2^h}\) stabiliser codes.
        \item Quantum sets of \(n\) symplectic polar spaces of rank \(h\) in PG\((h(n-k)-1,2)\) with minimum distance \(d\).
    \end{enumerate}
    Moreover, stabiliser codes over fields of even order are equivalent if and only if their quantum sets of symplectic polar spaces are projectively equivalent.
\end{theorem}
\begin{proof}
    This follows from \rref{Theorem}{thm:maingeom}, \rref{Theorem}{thm:qsohlequiv} and \rref{Theorem}{thm:qsosps}.
\end{proof}

If \(h=1\), a symplectic polar space of rank \(h\) is a line and we get \rref{Theorem}{thm:quantumlines} (a symplectic transformation of a line is a projective transformation of the line).

Note that, if a code is pure, then the minimum distance of the corresponding set of symplectic polar spaces is minimum size of a set of dependent points in distinct \((2h-1)\)-spaces of \(X\).

%And then this will be a corollary of the main thm of this section:

%\begin{theorem}
%    Quantum \(h\)-sets of lines from equivalent codes have the same ``hyperplane intersections'' distribution.
%\end{theorem}
%\begin{proof}
%    A bit like Thm 3.3 or 3.4 in \cite{qeccandtheirgeom}?
%    Take an arbitrary \(S\in\S\) and let \(v\) be a vector such 
%\end{proof}

\section{Applications}\label{sec:applications}

In this section, we %use \rref{Theorem}{thm:final} to
prove the uniqueness of a \([\![4,0,3]\!]_4\) code and the nonexistence of a \([\![7,1,4]\!]_4\) code and an \([\![8,0,5]\!]_4\) code.

Note that, as stabiliser codes over \(\ZZ/4\ZZ\) (see \rref{Remark}{remark:pauli}), codes with these parameters do not exist, since they would imply the existence of a \([\![4,0,3]\!]_2\), a \([\![7,1,4]\!]_2\) and an \([\![8,0,5]\!]_2\) code, which do not exist \cite{codetables}.

Our approach is the following. By \rref{Theorem}{thm:final}, an \([\![n,n-2(d-1),d]\!]_4\) code is equivalent to a set of \(n\) solids (projective \(3\)-spaces) in \(\PG(4d-5,2)\), each equipped with a symplectic polarity, such that every \(d-1\) solids span the entire space, and such that every codimension two space intersects an even number of solids in a hyperbolic line. Note that the minimum distance being \(d\), is equivalent to every \(d-1\) solids spanning the whole space. We do not have the more elaborate definition of minimum distance because the code is pure, as it meets the quantum Singleton bound in \rref{Theorem}{thm:mds}.

For \(d=3\), we classify sets of \(n\) solids in \(\PG(7,2)\), for \(n\in\{4,5,6\}\). We then look for symplectic polarities with the property that every codimension two subspace intersects an even number of these solids in a hyperbolic line. This allows us to prove that a \([\![4,0,3]\!]_4\) is unique and to determine all \([\![6,2,3]\!]_4\) codes. We observe that it is not feasible computationally for \(d=4\) and \(n=7\). However, we use a projection argument to rule out the \([\![7,1,4]\!]_4\) code, see \rref{Section}{sec:nonexist}.

\subsection{There is a unique \texorpdfstring{\([\![4,0,3]\!]_4\)}{[[4,0,3]]4} code}

It was proved in \cite{led} that the \([\![4,0,3]\!]_4\) stabiliser code is unique. Here, we reprove this, and demonstrate how our method applies to these parameters.

By \rref{Theorem}{thm:final}, a \([\![4,0,3]\!]_4\) code is equivalent to a set of four solids (projective \(3\)-spaces) in \(\PG(7,2)\), each equipped with a symplectic polarity, such that every two solids span the entire space, and such that every codimension two space intersects an even number of solids in a hyperbolic line.

We first show that there are only three configurations of four solids in \(\PG(7,2)\) such that every two solids span the entire space.
Without loss of generality, the four solids are spanned by the four column blocks of a block matrix of the form
\[\left(\begin{array}{@{}c|c|c|c@{}}
    I&O&I&I\\
    %\hline
    O&I&I&A
\end{array}\right).\]
Indeed, we can first perform row operations (i.e.\ projective transformations) to bring it into the form
\[\left(\begin{array}{@{}c|c|c|c@{}}
    *&*&*&*\\
    %\hline
    *&*&*&*
\end{array}\right)\mapsto\left(\begin{array}{@{}c|c|c|c@{}}
    I&O&*&*\\
    %\hline
    O&I&*&*
\end{array}\right)\]
and then perform column operations within each block (which do not change the span of the columns) to get
\[\mapsto\left(\begin{array}{@{}c|c|c|c@{}}
    I&O&I&I\\
    %\hline
    O&I&*&*
\end{array}\right)\mapsto\left(\begin{array}{@{}c|c|c|c@{}}
    I&O&I&I\\
    %\hline
    O&*&I&*
\end{array}\right)\mapsto\left(\begin{array}{@{}c|c|c|c@{}}
    I&O&I&I\\
    %\hline
    O&I&I&*
\end{array}\right).\]
Now, \(A\) is still determined up to a conjugation \(A\mapsto XAX^{-1}\), as we can do
\[\left(\begin{array}{@{}c|c|c|c@{}}
    I&O&I&I\\
    %\hline
    O&I&I&A
\end{array}\right)\mapsto\left(\begin{array}{@{}c|c|c|c@{}}
    X&O&X&X\\
    %\hline
    O&X&X&XA
\end{array}\right)\mapsto\left(\begin{array}{@{}c|c|c|c@{}}
    I&O&I&I\\
    %\hline
    O&I&I&XAX^{-1}
\end{array}\right)\]
and up to a permutation of the four solids. Together with the condition that every two column blocks form an invertible matrix, one can check by computer that there are only three different configurations (five up to conjugation and three up to conjugation and permutation):
\[A_1=\begin{pmatrix}1&1&0&0\\0&0&1&0\\0&0&0&1\\1&0&0&0\end{pmatrix},\quad A=\begin{pmatrix}0&1&1&0\\1&0&0&0\\0&0&0&1\\0&1&0&0\end{pmatrix}\text{\quad and\quad}A_3=\begin{pmatrix}1&1&0&0\\1&0&0&0\\0&0&1&1\\0&0&1&0\end{pmatrix}.\]

For each of them, we check whether they can be equipped with a symplectic polarity such that every codimension two space intersects an even number of solids in a hyperbolic line. Suppose that this is the case. For each line \(\ell\) in each of the four solids (so \(4\cdot35=140\) lines), let \[x_\ell=\begin{cases}0\text{ if \(\ell\) is totally isotropic}\\1\text{ if \(\ell\) is hyperbolic}\end{cases}.\]
Then for each codimension two space \(\pi\), the sum of the \(x_\ell\)'s for which \(\ell\) is the intersection of \(\pi\) with one of the four solids, equals zero modulo \(2\). So it suffices to check a system of homogeneous equations over \(\FF_2\). This can be done by computer. Given such a solution, one still has to check whether the lines \(\ell\) for which \(x_\ell=1\), really are hyperbolic lines of a symplectic polarity. Now, \(W(3,2)\) contains \(15\) totally isotropic lines and \(20\) hyperbolic lines, so we first of all check that \(|\{\ell\in\pi\mid x_\ell=1\}|=20\). Then, we choose a (potential) hyperbolic line \(\ell_1\) and loop over all (potential) hyperbolic lines \(\ell_2\) to check whether \(\ell_2=\ell_1^\perp\) (we can fix \(\ell_1\) by \rref{Lemma}{lemma:witt}). This is true if and only if every other (potential) hyperbolic line intersects exactly one of them, by \rref{Theorem}{thm:qsosps}. If we find such a line \(\ell_2\), then these lines are indeed the hyperbolic lines of a symplectic polarity (and it is unique by \rref{Lemma}{lemma:lines}).

Implementing these constraints on a computer results in three solutions for \(A_3\) and no solutions for \(A_1\) and \(A_2\). However, the three solutions can be checked to be equivalent up to a projective transformation that fixes the solids.

The code for this can be found on \url{https://github.com/robinsimoens/stabiliser-codes-over-fields-of-even-order}. We used the mathematics software system Sage \cite{sage}.

%Note that this solution is necessarily unique. Let \(\pi_1,\dots,\pi_4\) be the four solids. Each of the \(15\) planes in \(\pi_1\) is contained in a unique codimension two space that intersects \(\pi_2\) and \(\pi_3\) in a plane as well (namely, the intersection of the hyperplane spanned by the plane and \(\pi_2\), and the hyperplane spanned by the plane and \(\pi_3\)). It must intersect \(\pi_4\) in a line (a plane is impossible since otherwise the codimension two space would be contained in \(4\) different hyperplanes) which is necessarily totally isotropic. So just the configuration of the four solids already determines the \(15\) totally isotropic lines of \(\pi_4\) \robin{Not true after all, might be twice the same line}, and hence the symplectic polarity on \(\pi_4\), uniquely.

\subsection{A \texorpdfstring{\([\![7,1,4]\!]_4\)}{[[7,1,4]]4} code does not exist}\label{sec:nonexist}

By \rref{Theorem}{thm:final}, a \([\![7,1,4]\!]_4\) code is equivalent to a set of seven solids (projective \(3\)-spaces) in \(\PG(11,2)\), each equipped with a symplectic polarity, such that every three solids span the entire space, and such that every codimension two space intersects an even number of solids in a hyperbolic line.
(Recall that the minimum distance being four, is equivalent to every three solids spanning the whole space. We do not have the more elaborate definition of minimum distance because the code is pure, as a \([\![7,1,4]\!]_4\) code meets the quantum Singleton bound in \rref{Theorem}{thm:mds}.)

\begin{lemma}\label{lemma:project}
    Let \(\X\) be a quantum set of \(n\) symplectic polar spaces of rank \(h\) in \(\PG(r-1,2)\) with minimum distance \(d\). If we project from one of its spaces, we get a quantum set of \(n-1\) symplectic polar spaces of rank \(h\) in \(\PG(r-2h-1,2)\) with minimum distance \(d-1\) or \(d\).
\end{lemma}
\begin{proof}
    By definition.
\end{proof}

In particular, if there exists an \([\![n,k,d]\!]_{2^h}\) code, then there also exists an \([\![n-1,k+1,d-1]\!]_{2^h}\) code. For this reason, we first generate all \([\![6,2,3]\!]_4\) codes. We do this in a similar way as in the previous section. We may assume that the six solids are spanned by the column blocks of a block matrix of the form
\[\left(\begin{array}{@{}c|c|c|c|c|c@{}}
    I&O&I&I&I&I\\
    \hline
    O&I&I&A&B&C
\end{array}\right).\]
A computer program tells us that, up to conjugation and permutation, there are \(341\) configurations of six solids in \(\PG(7,2)\) such that every two solids span the entire space. This was done by first determining all \(A\)-blocks up to conjugation and permutation, then all \(B\)-blocks given these \(A\)-blocks, and then the \(C\)-blocks.

Then, we equip each solid with a symplectic polarity such that every codimension two space is skew to an even number of hyperbolic lines. In a similar way as for the previous example, we find 1311 solutions for this. So there are at most 1311 different \([\![6,2,3]\!]_4\) codes.

Now, to rule out the existence of a \([\![7,1,4]\!]_4\) code, we may assume that its solids are spanned by a matrix of the form
\[\left(\begin{array}{@{}c|c|c|c|c|c|c@{}}
    X_0&O&O&Y&A_0&B_0&C_0\\
    \hline
    O&X_1&O&Y&A_1&B_1&C_1\\
    \hline
    O&O&X_2&Y&A_2&B_2&C_2
\end{array}\right)\]
where \(A_0,A_1,B_0,B_1,C_0\) and \(C_1\) are solutions to the \([\![6,2,3]\!]_4\) codes found before, and such that every two columns span a hyperbolic line and its perpendicular line. In particular, every two rows of
\[\left(\begin{array}{@{}c|c|c|c@{}}
    Y&A_2&B_2&C_2
\end{array}\right)\]
have to be symplectic orthogonal to the rows of the matrix
\[\left(\begin{array}{@{}c|c|c|c@{}}
    Y&A_0&B_0&C_0\\
    \hline
    Y&A_1&B_1&C_1
\end{array}\right)\]
with respect to the symplectic form
\[f(v,w)=v_1w_2+v_2w_1+\cdots+v_{15}w_{16}+v_{16}w_{15}.\]

Moreover, every submatrix of \(3\times3\) blocks is invertible. This implies extra conditions on \(A_2,B_2\) and \(C_2\). 
For every solution for \(A_2,B_2\) and \(C_2\) obtained by computer, we observe that the rows of \(\left(\begin{array}{@{}c|c|c|c@{}}
    Y&A_2&B_2&C_2
\end{array}\right)\) are pairwise symplectic orthogonal. This implies that every two rows of \(X_2\) must be symplectic orthogonal. This is a contradiction, since \(X_2\) has full rank. So there does not exist a \([\![7,1,4]\!]_4\) code.

The code for this can be found on \url{https://github.com/robinsimoens/stabiliser-codes-over-fields-of-even-order}. We used the mathematics software system Sage \cite{sage}. We also wrote the same program in GAP \cite{gap}, using in particular the FinInG package \cite{fining}, to check the validity of the program. It gave the same results.

By \rref{Lemma}{lemma:project}, the nonexistence of a \([\![7,1,4]\!]_4\) code implies the nonexistence of an \([\![8,0,5]\!]_4\) code.

\section{Conclusion}

We showed that every stabiliser code over a field of even order can be seen as a binary stabiliser code. Geometrically, they correspond to quantum sets of symplectic polar spaces, up to projective equivalence. This allowed us to prove the uniqueness of a \([\![4,0,3]\!]_4\) code and the nonexistence of a \([\![7,1,4]\!]_4\) code and an \([\![8,0,5]\!]_4\) code.

We present some open problems and directions for future work:

\begin{enumerate}[(i)]
    \item It is known that one-dimensional binary stabiliser codes can be seen as graphs (hence they are also called \emph{graph states}). Moreover, equivalent codes correspond to graphs up to \emph{local complementation}. The notion of local complementation extends to stabiliser codes over fields of prime order, see \cite{beigi}. What about stabiliser codes over fields of even order? They can be seen as graphs where we group the vertices into \(n\) sets of \(h\) vertices. The question remains what a conjugation with a Clifford operator translates to on these graphs.
    \item We solved \cite[Research Problem~4]{qeccandtheirgeom} in the case when \(q\) is even. Can we do something similar for codes over finite fields of odd characteristic? There are two drawbacks. First, there does not always exist a trace-orthogonal basis in odd characteristic. However, one can circumvent this using a so-called polarisation of the phase space, see \cite{gross,heinrich}. Second, it has not been established what these codes correspond to geometrically. Multiplying a column with \(-1\) does not affect the projective geometry, but it does not come from a conjugation with a Clifford operator. This suggests that, instead of looking at the projective points determined by the columns, we should rather look at them as affine points.
    \item We can do something similar for stabiliser codes over modular rings, reducing them to binary codes with certain properties. %However, there is no neat trick like the trace-orthogonal basis that we know of.
\end{enumerate}

\subsection*{Acknowledgements}
We would like to thank the referees, Markus Grassl and Markus Heinrich for their useful comments.

%Simeon Ball is supported by PID2020-113082GB-I00 financed by MCIN / AEI / 10.13039/501100011033, the Spanish Ministry of Science and Innovation.

%Edgar Moreno is partially supported by the Higher Interdisciplinary  Training Center (CFIS) - Universitat Politècnica de Catalunya (UPC) and SANTANDER- UPC INITIAL RESEARCH GRANTS (INIREC).

%Robin Simoens is supported by the Research Foundation Flanders through the grant 11PG724N.

\newpage

\appendix
\renewcommand{\thesection}{Appendix \Alph{section}}

\section{An algebraic proof of \texorpdfstring{\rref{Theorem}{thm:cliffordequiv}}{Theorem 30}}\label{app:skolem}

\rref{Theorem}{thm:cliffordequiv} was proved in \cite{cliffordequiv} more generally for stabiliser codes over \(\ZZ/d\ZZ\).
In this appendix, we provide an alternative, more algebraic proof instead of a constructive proof.

Let \(\tau\) be the map
\begin{align*}
    \tau&:\P_n%/{\{\pm1,\pm i\}}
    \to\mathbb{F}_2^{2n}:\\
    &cX(a_1)Z(b_1)\otimes\cdots\otimes X(a_n)Z(b_n)\mapsto(a_1,\dots,a_n,b_1,\dots,b_n).
\end{align*}
as defined in \rref{Definition}{def:tau} (with \(q=2\)).
Recall that a symplectic transformation on \(\FF_2^{2n}\) is a change of basis that preserves the symplectic inner product.

\begin{theorem}[\cite{cliffordequiv}]
    Binary stabilisers are conjugated by a Clifford operator if and only if the corresponding symplectic self-orthogonal codes are the same up to a symplectic transformation.
\end{theorem}

%\begin{center}\begin{tikzcd}
%    \mathcal{S} \arrow[r, "S\mapsto U^\dagger SU"] \arrow[d, "\tau"] &[20mm] \mathcal{S}' \arrow[d, "\tau"]\\
%    C \arrow[r, "v\mapsto Av"] & C'
%\end{tikzcd}\end{center}

\begin{proof}
    Let \(\S_1\) and \(\S_2\) be two binary stabilisers and let \(C_1=\tau(\S_1)\) and \(C_2=\tau(\S_2)\) be the corresponding symplectic self-orthogonal codes.
    
    First suppose that \(\S_1\) and \(\S_2\) are conjugated by any linear operator \(U\), i.e.\ \(\S_2=U\S_1U^{-1}\).
    Define the bijection \(\sigma:C_1\to C_2\) that sends \((a|b)\) to \(\tau(U\tau^{-1}(a|b)U^{-1})\). It is well-defined because conjugation preserves phase factors. It is linear %, since \(\tau\) is a group morphism
    by \rref{Lemma}{lemma:tau}.
    %\footnote{Note that this step only works if {\color{red}\(h=1\)} because in general, \(\tau(U^\dagger\tau^{-1}(\alpha v)U)\neq\alpha\tau(U^\dagger\tau^{-1}(v)U)\). For example, if \(\mathbb{F}_q=\mathbb{F}_4=\{0,1,\alpha,\alpha^2\}\) and \(U=\mathrm{diag}(1,-i,-1,-i)\), then \(U^\dagger X(1)U=iX(1)Z(1)\) but \(U^\dagger X(\alpha)U=-X(\alpha)\neq iX(\alpha)Z(\alpha)\). TODO wrong counterexample}
    %So there exists an invertible matrix \(A\) such that \(\sigma(v)=Av\).
    Choose arbitrary \((a|b),(a'|b')\in\mathbb{F}_2^{2n}\). Let \(E_1,E_2\in\mathcal{P}_n\) with \(\tau(E_1)=(a|b)\) and \(\tau(E_2)=(a'|b')\). Then by the proof of \rref{Lemma}{lemma:commute},
    \[E_1E_2=(-1)^{\langle (a|b),(a'|b')\rangle_s}E_2E_1.\]
    On the other hand, \(\sigma(a|b)=\tau(UE_1U^{-1})\) and \(\sigma(a'|b')=\tau(UE_2U^{-1})\), so
    \begin{align*}
        E_1E_2&=U^{-1}\left(UE_1U^{-1}\right)\left(UE_2U^{-1}\right)U\\
        &=U^{-1}(-1)^{\langle\sigma(a|b),\sigma(a'|b')\rangle_s}\left(UE_2U^{-1}\right)\left(UE_1U^{-1}\right)U\\
        &=(-1)^{\langle \sigma(a|b),\sigma(a'|b')\rangle_s}E_2E_1,
    \end{align*}
    so \(\langle (a|b),(a'|b')\rangle_s=\langle \sigma(a|b),\sigma(a'|b')\rangle_s\) and thus, \(\sigma\) is a symplectic transformation.
    
    Conversely, assume that \(C_1\) and \(C_2\) are the same up to a symplectic transformation \(\sigma\). Consider the map \(g:\S_1\to\S_2\) defined by \(g=\tau^{-1}\circ\sigma\circ\tau\). Its image is a class of four stabilisers that are the same up to a phase factor. We therefore define a new map \(h\) for which \(g(E)=[h(E)]\) as follows.
    Denote \(X_i\) resp.\ \(Z_i\) for the weight one Pauli operator with an \(X\) resp.\ \(Z\) in the \(i\)-th position %:
    %\[X_i=I\otimes\cdots\otimes\underbrace{X}_\text{\(i\)-th position}\otimes\cdots\otimes I\]
    %\[Z_i=I\otimes\cdots\otimes\underbrace{Z}_\text{\(i\)-th position}\otimes\cdots\otimes I\]
    and let \(W=\{ X_1,\dots,X_n,Z_1,\dots,Z_n\}\).
    Note that it spans about half the Pauli group. For every \(X_i\), choose and fix its image \(h(X_i)\) in \(g(X_i)\) such that it has order two again (and similar for the \(Z_i\)'s).
    There is a unique way to extend this map to a homomorphism of the group generated by \(W\), i.e.\ such that \(h(E_1E_2)=h(E_1)h(E_2)\). Indeed, the product relations between the \(X_i\)'s and \(Z_i\)'s are preserved because
    \[h(X_i)^2=h(Z_i)^2=I\]
    and
    \[h(E_1)h(E_2)=(-1)^{\langle \sigma(\tau(E_1)),\sigma(\tau(E_2))\rangle_s}h(E_2)h(E_1)=(-1)^{\langle\tau(E_1),\tau(E_2)\rangle_s}h(E_2)h(E_1)\]
    for all \(E_1,E_2\in W\), where we used that \(\sigma\) preserves the symplectic inner product.
    In particular, \(g(E)=[h(E)]\) for every such operator, since \(g(E_1E_2)=g(E_1)g(E_2)\) by \rref{Lemma}{lemma:tau}. Since \(W\) forms a basis for \(\mathcal{H}\), we can extend \(h\) uniquely to a linear map on \(\mathcal{H}\) that fulfils \(h(E_1+E_2)=h(E_1)+h(E_2)\), which makes \(h\) an algebra morphism. We can now apply the Skolem-Noether theorem to conclude that \(h(E)=UEU^{-1}\) for some \(U\).
    Choose an arbitrary Pauli operator \(E\in\mathcal{P}_n\). Then \(h(E)\) is again a Pauli operator, so \(h(E)h(E)^\dagger=I\). Therefore, \(UU^\dagger\) commutes with \(E\), and because \(E\) was chosen arbitrarily, it commutes with all Pauli operators. Hence, \(UU^\dagger=cI\) for some constant \(c\). Now, \(UU^\dagger\) is positive semidefinite, so \(c\) is real positive and we can replace \(U\) by \(\frac{1}{\sqrt{c}}U\) to obtain the desired Clifford operator.
\end{proof}

%The above proof 
%\robin{Question:
%Given a Clifford \(U\), is its action completely determined by \(U^\dagger X(1)U\) and \(U^\dagger Z(1)U\)?
%Similar question: does the theorem still hold for nonbinary case? In which the above question would be answered yes.}

%\robin{Can the proof be generalized to \(\mathbb{Z}/p\mathbb{Z}\)?
Note that the above proof can be generalised easily to stabiliser codes over any prime field \(\FF_p\) or, equivalently, stabiliser codes over \(\ZZ/p\ZZ\) (we refer to \cite{ketkar} and \cite{cliffordequiv} for their respective definitions). %However, for stabiliser codes over arbitrary fields \(\FF_q\) or modular rings \(\ZZ/q\ZZ\), with \(q\) not a prime number, this proof does not work.

%%%%%%%%%%%%%%%%%%%%%%%%%%%%%%%%%%

\vfill
\textsc{\small Simeon Ball}\\[-1mm]
\textsc{\small Department of Mathematics}\\[-1mm]
\textsc{\small Universitat Politècnica de Catalunya}\\[-1mm]
\textsc{\small C. Pau Gargallo 14, 08028 Barcelona, Spain}\\[-1mm]
%{\small M\`odul C3, Campus Nord,}\\
%{\small c/ Jordi Girona 1-3,}\\
%{\small 08034 Barcelona, Spain} \\
{\it E-mail address:} {\tt simeon.michael.ball@upc.edu}\\[-1mm]

\textsc{\small Edgar Moreno}\\[-1mm]
\textsc{\small Higher Interdisciplinary Training Center (CFIS)}\\[-1mm]
\textsc{\small Universitat Politècnica de Catalunya}\\[-1mm]
\textsc{\small C. Pau Gargallo 14, 08028 Barcelona, Spain}\\[-1mm]
{\it E-mail address:} {\tt edgar.moreno.martinez@estudiantat.upc.edu}\\[-1mm]

\textsc{\small Robin Simoens}\\[-1mm]
\textsc{\small Department of Mathematics: Analysis, Logic and Discrete Mathematics}\\[-1mm]
\textsc{\small Ghent University}\\[-1mm]
\textsc{\small Krijgslaan 281, 9000 Gent, Belgium}\\[-1mm]
\textsc{\small Department of Mathematics}\\[-1mm]
\textsc{\small Universitat Politècnica de Catalunya}\\[-1mm]
\textsc{\small C. Pau Gargallo 14, 08028 Barcelona, Spain}\\[-1mm]
{\it E-mail address:} {\tt Robin.Simoens@UGent.be}

\end{document}